\documentclass[11pt]{amsart}
\usepackage{MnSymbol}
\usepackage{mathrsfs}
\usepackage[utf8]{inputenc}
\usepackage{amsmath,amsthm}
\usepackage{tikz-cd}
\usetikzlibrary{matrix}
\usepackage{graphicx}
\usepackage{wrapfig}

\newtheorem{Theorem}{Theorem}[section]
\newtheorem{Corollary}[Theorem]{Corollary}
\newtheorem{Lemma}[Theorem]{Lemma}
\newtheorem{Proposition}[Theorem]{Proposition}
\theoremstyle{definition}
\newtheorem{Remark}[Theorem]{Remark}

\newtheorem{Definition}[Theorem]{Definition}

\numberwithin{equation}{section}

\DeclareMathAlphabet\mathbb{U}{msb}{m}{n}
\newcommand{\mono}{\rightarrowtail}
\newcommand{\epi}{\twoheadrightarrow}

\def\AA{{\mathcal A}}

\def\ZZ{{\mathbb Z}}
\def\OO{{\mathcal O}}
\def\CC{{\mathcal C}}
\def\BB{{\mathcal B}}
\def\DD{{\mathcal D}}
\def\MM{{\mathcal M}}
\def\NN{{\mathcal N}}
\def\FF{{\mathcal F}}

\def\NNN{{\sf N\hspace{1pt}}}
\def\colim{{\rm colim}}

\usepackage[all]{xy}
\begin{document}

\title{Higher colimits, derived functors and homology}
\author{Sergei O. Ivanov}

\address{
Laboratory of Modern Algebra and Applications,  St. Petersburg State University, 14th Line, 29b,
Saint Petersburg, 199178 Russia}\email{ivanov.s.o.1986@gmail.com}

\author{Roman Mikhailov}
\address{Laboratory of Modern Algebra and Applications, St. Petersburg State University, 14th Line, 29b,
Saint Petersburg, 199178 Russia and St. Petersburg Department of
Steklov Mathematical Institute} \email{rmikhailov@mail.ru}

\author{Vladimir Sosnilo}
\address{Laboratory of Modern Algebra and Applications, St. Petersburg State University, 14th Line, 29b,
Saint Petersburg, 199178 Russia}
\email{vsosnilo@gmail.com}

\thanks{The project is supported by the grant of the Government of the Russian Federation
for the state support of scientific research carried out under the supervision of leading scientists,
agreement 14.W03.31.0030 dated 15.02.2018}

\begin{abstract}
A theory of higher colimits over categories of free presentations is developed. It is shown that different homology functors such as Hoshcshild and cyclic homology of algebras over a field of characteristic zero, simplicial derived functors, and group homology can be obtained as higher colimits of simply defined functors. Connes' exact sequence linking Hochschild and cyclic homology was obtained using this approach as a corollary of a simple short exact sequence. As an application of the developed theory it is shown that the third reduced $K$-functor can be defined as the colimit of the second reduced $K$-functor applied to the fibre square of a free presentation of an algebra. A Hopf-type formula for odd dimensional cyclic homology of an algebra over a field of characteristic zero is also proved. 
\end{abstract}

\maketitle

\section{\bf Introduction}

For an "algebraic object" $A$ (for example, a group, an abelian group, an associative algebra, ...), denote
by ${\sf Pres}(A)$ the category of presentations of $A$ as a quotient of a free (or projective) object $F\twoheadrightarrow A$.
In \cite{IvanovMikhailov}, the authors developed the theory of derived limits over categories of free presentations and showed that many functors can be obtained as (derived) limits of certain simply defined functors ${\sf Pres}(A)\to ( {\sf modules\ over\ a\ commutative\ ring} )$. For a functor $\mathcal F$ from ${\sf Pres}(A)$ to the category of modules over a fixed commutative ring and any $c\in {\sf Pres}(A)$, there is a natural injection ${\rm lim}\: \mathcal F\hookrightarrow \mathcal F(c)$ and ${\rm lim}\: \mathcal F$ is the largest submodule of $\mathcal F(c)$ that does not depend on $c$. In other words, ${\rm lim}\: \mathcal F$ is the largest constant sub-functor of $\mathcal F.$ This approach gives a method to construct various functors. For example, for a group $G$ and the category of free presentations $F\twoheadrightarrow G$, there is a natural description of the first derived functor (in the sense of Dold-Puppe) of the symmetric cube $L_1S^3$ applied to the abelianization of $G$ as $L_1S^3(G_{ab})={\rm} {\rm lim}\frac{[F,F,F]}{[R,R,F][F,F,F,F]},$ where $R={\sf Ker}(F\epi G)$  [here we use the left-normalized notation for commutator subgroups] (see \cite{MP2016}). 
One can write any expression that functorially depends on a pair $(F,R)$ and apply the (derived) limit over ${\sf Pres}(G)$. 
Here $R$ and $F$ are used as bricks in the constructor. 
Group homology, Hochschild and cyclic homology of algebras, and certain derived functors can be obtained in this way \cite{IvanovMikhailov}, \cite{MP2016}, \cite{MP2017}.

This paper concerns (higher, derived) colimits over the categories of presentations. In the above notation, for any $c\in {\sf Pres}(A)$, we get a natural surjection $\mathcal F(c)\twoheadrightarrow{\rm colim}\: \mathcal F$ and ${\rm colim}\: \mathcal F$ is the largest quotient of $\mathcal F(c)$ that does not depend on $c$. In other words, ${\rm colim}\: \mathcal F$ is the largest constant quotient-functor of $\mathcal F.$ A simple illustration of the above statement about the {\it largest independent quotient} is the following. For an algebra $A$ and the category of non-unital free presentations $F\twoheadrightarrow A,$ one can obtain the first cyclic homology of $A$ as follows
$$
{\rm colim}\ R\cap [F,F]=HC_1(A),
$$
where $R={\sf Ker}(F\epi G).$
The needed largest quotient independent of the presentation is given by the Hopf formula: $R\cap [F,F]\twoheadrightarrow\frac{R\cap [F,F]}{[R,F]}=HC_1(A).$

At first glance, the theory of derived colimits looks like a mirror-dual analog of the theory of derived limits developed in the papers of the authors. In a sense, this is true, however, there is no way to transfer the principal methods from limits to colimits because the category ${\sf Pres}(A)$ is not self-dual. For example, for any $\mathcal F$ and $c\in {\sf Pres}(A)$, ${\rm lim}\: \mathcal F$ is the equalizer of two maps $\mathcal F(c)\rightrightarrows \mathcal F(c\sqcup c)$, where $\sqcup$ is the coproduct in ${\sf Pres}(A)$ (see section 2 in \cite{MP2017}). In the case of colimit, we can not present ${\rm colim}\: \mathcal F$ as a coequalizer of two such maps because $c\times c$ does not exist in ${\sf Pres}(A).$ We have to coequalize all the maps $\mathcal F(c')\to \mathcal F(c)$ for $c'\to c$ from ${\sf Pres}(A)$. As a corollary, we don't have a simple criterion of triviality of colimits, as we have in the case of limits (the property called {\it monoadditivity} in \cite{IvanovMikhailov}). Unlike the case of limits, the additive functors (with respect to coproducts in the category of representations) can give nontrivial colimits in all degrees. For example, for a group $G$ and the category of free presentations $R\rightarrowtail F\twoheadrightarrow G$, ${\rm lim}_n\: F_{ab}=0$ for all $n\geq 0$ (see \cite{IvanovMikhailov}), however, $${\rm colim}_n\: F_{ab}=H_{n+1}(G),$$ for $n\geq 0$ (Theorem \ref{thm_group_homology}).

In general, colimits can be viewed as a generalization of derived functors of non-additive functors in the following sense. Let $\mathcal C$ be a category with enough of projectives and $c\in \mathcal C.$ Denote by ${\sf Pres}(c)$  the category of effective epimorphisms $p\epi c,$ where $p$ is a projective object. Assume that $\varepsilon_\bullet: p_\bullet\to c$ is a simplicial projective resolution of $c$ (see Definition \ref{def_simpl_proj_res}) and $\varepsilon^{\sf P}_\bullet\in {\sf Pres}(c)^{\Delta^{\rm op}}$ is the corresponding simplicial presentation. Then for any ``good enough'' functor $\Phi:{\sf Pres}(c)\to {\sf Mod}(\Lambda)$ to the category of modules over some ring $\Lambda$  there is an isomorphism 
$$ \colim_n\ \Phi= \pi_n( \Phi( \varepsilon^{\sf P}_\bullet ) )$$
(see Theorem \ref{th_derived} for details). In particular, colimits give a way to define derived functors of non-additive functors not using simplicial resolutions (see Corollary \ref{cor_derived}). For example, given a group $G$   
and a functor $\Phi: ({\sf Groups})\to ({\sf Abelian\ groups}),$ the derived colimits ${\rm colim}_n\: \Phi(F)$ over the category of free presentations $R\rightarrowtail F\twoheadrightarrow G$ coincide with simplicial derived functors of $\Phi:$
$${\rm colim}_n\: \Phi(F)=L_n\Phi(G).$$
 (see Proposition \ref{derpropos}). Another simple example of usage of  Theorem \ref{th_derived} is the following formula for Andr\'e-Quillen homology of a commutative $k$-algebra $A$ with coefficients in an $A$-module $M$ (Proposition \ref{prop_Andre-Quillen})
\begin{equation}\label{eq_andre_quillen}
D_n(A/k,M)=\colim_n\: \Omega^{\rm comm}(F)\otimes_F M.
\end{equation} 
Here the derived colimit is taken over the category surjective homomorphisms $F\epi A,$ where $F$ is the (commutative) polynomial $k$-algebra; and $\Omega^{\rm comm}(F)$ is the module of K\"ahler differentials of $F$. 

The main results of this paper are the following.
\begin{enumerate}
\item A generalization of the description of the group homology mentioned above for the case of non-trivial coefficients. Let $G$ be a group, $M$ a $\mathbb Z[G]$-module $M$. For $n\geq 0$, (Theorem \ref{thm_group_homology})
$$H_{n+1}(G,M)= \colim_n\ H_1(F,M),$$
where the colimits are taken over the category of free presentations $F\twoheadrightarrow G$.

\item For the category of unital associative algebras over any field, and the category of free presentations $F\twoheadrightarrow A$, there is the following description of Hochschild homology with coefficients in an $A$-bimodule $M$ (Theorem  \ref{theorem_hochsch})
$$H_{n+1}(A,M)=\colim_n\: H_1(F,M)=\colim_n\: \Omega(F)\otimes_{F^e} M,$$
$$HH_{n+1}(A)=\colim_n\: \Omega(F)_\natural$$
for $n\geq 1,$
where $\Omega(F):={\rm Ker}(F\otimes F \epi F)$ is the bimodule of non-commutative differential forms of $F,$ $F^e=F^{op}\otimes F$ and $$M_\natural:=H_0(F,M)=M/[M,F].$$ 
(Compare the formula $ H_{n+1}(A,M)=\colim_n\  \Omega(F)\otimes_{F^e} M$ with the formula  \eqref{eq_andre_quillen}.)

\item For the category of unital associative algebras over a field of characteristic zero, and the category of free presentations $F\twoheadrightarrow A$, there is the following description of Hochschild and reduced cyclic homology (see Proposition \ref{connesseq}):
\begin{align*}
& \overline{HC}_{n+1}(A)={\rm colim}_n[F,F],\ n\geq 1,\\
& \overline{HC}_{n+3}(A)={\rm colim}_n[R,R],\ n\geq 1,\\
& HH_{n+2}(A)={\rm colim}_n \frac{[F,F]}{[R,R]},\ n\geq 2.
\end{align*}
The natural short exact sequence
$$ 0 \longrightarrow [R,R] \longrightarrow [F,F] \longrightarrow \frac{[F,F]}{[R,R]} \longrightarrow 0   $$
gives rise to the long exact sequence of (derived) colimits
\begin{multline*} \dots \to \overline{HC}_6(A) \to \overline{HC}_4(A) \to HH_5(A) \to \overline{HC}_5(A) \to\\  \overline{HC}_3(A) \to HH_4(A) \to \overline{HC}_4(A)\to \overline{HC}_2(A).
\end{multline*}
Moreover, there is an isomorphism 
$$ \overline{HC}_n(A)=\colim_n\: F_\natural,$$ 
for $n\geq 1$ (Theorem \ref{theorem_cyclic}) and the short exact sequence
$$ 0 \longrightarrow \bar F_\natural \longrightarrow \Omega(F)_\natural \longrightarrow [F,F] \longrightarrow 0$$
(Lemma \ref{lemma_Omega(F)_decomposition}) induces   
the Connes'-like long exact sequence as above (Proposition \ref{prop_connes_omega}).

\item Let $k$ be a noetherian regular commutative ring and $A$ a $k$-algebra. Then
$$
{\rm colim}\ \tilde K_2(F\times_AF)=\tilde K_3(A),
$$
where the colimit is taken over the category of free presentations $F\twoheadrightarrow A$ of $k$-algebras (Proposition \ref{colimK2}(1)). Here $\tilde K_n$ is the reduced $K$-functor ($n$th homotopy group of the homotopy fiber of the map of spectra $K(k)\to K(A)$), $F\times_AF$ is the fibred square of the epimorphism $F\twoheadrightarrow A$. Analogously, for a non-unital ring $A$,
$$
{\rm colim}\ K_2(F\times_AF)=K_3(A),
$$
where the colimit is taken over the category of free presentations of non-unital rings (Proposition \ref{colimK2}(2)).
\end{enumerate}

We also prove a Hopf-type formula for the odd dimensional cyclic homology of algebras over a field $k$ of characteristics zero (Theorem \ref{th_hopf}): for a non-unital free presentation of an algebra $A$, $F\twoheadrightarrow A$, and $n\geq 0$, there is a natural isomorphism
$$ HC_{2n+1}(A)=\frac{R^{n+1}\cap [F,F]}{[R,R^n]}.$$
Note that this isomorphism is strict, without (co)limits.

The paper is organized as follows. In Section 2, we consider general facts about colimits over categories. In particular, we show that, the colimit of a functor over a strongly connected category is the largest constant quotient (Proposition \ref{colimitintuitive}). In Section 3 we show that the simplicial derived functors can be presented as higher colimits. Section 4 is about groups, we show that the group homology with coefficients in any module can be defined as higher colimits of the first homology of the free group over the category of free presentations (Theorem \ref{thm_group_homology}). In Section 5 we give a description of the Hochschild and cyclic homology of unital algebras over fields of characteristics zero as higher colimits. As mentioned above, the Connes exact sequence which connects Hochschild and cyclic homology can be obtained as a sequence of higher colimits.
In the case of groups any functor that depends only on $R$ has trivial higher colimits (Proposition \ref{prop_Phi(R)}). In order to prove this we essentially use that a subgroup of a free group is free. However, for algebras it is not true. Moreover, the formula $\overline{HC}_{n+3}(A)=\colim_n\ [R,R]$ (Proposition \ref{prop_[F,F][R,F][R,R]}) holds. It can be informally  interpreted as the reason why the cyclic homology exists (and is different from Hochschild homology): because a subalgebra of a free algebra is not necessarily free. 
We also prove a Hopf-type formula for the odd dimensional cyclic homology of algebras over a field $k$ of characteristics zero (Theorem \ref{th_hopf_red}): for a free presentation of an algebra $A$, $F\twoheadrightarrow A$, and $n\geq 0$, there is a natural isomorphism
$ \overline{HC}_{2n+1}(A)=\frac{R^{n+1}\cap ([F,F]+k\cdot 1)}{[R,R^n]}.$
In Section 6, we give an analog of colimit formula and Hopf-type formulas for the cyclic homology of non-unital algebras. In Section 7, we consider $K_2$ and $K_3$-functors and give a description of the reduced $K_3$-functor as a colimit of reduced $K_2$-functors applied to the fibre squares.  

\section*{\bf Acknowledgments} The authors are grateful to Sergey Gorchinsky for discussions and helpful comments.

\section{\bf General facts about derived colimits over categories}

Let $\CC$ be a small category. Recall that its nerve is a simplicial set $\NNN\CC$ such that
$$(\NNN\CC)_0={\sf ob}(\CC),\hspace{1cm} (\NNN\CC)_1={\sf mor}(\CC), $$ the maps $d_0,d_1:(\NNN\CC)_1\to (\NNN\CC)_0$ and $s_0:(\NNN\CC)_0 \to (\NNN\CC)_1$ are defined as follows
$$ d_0(\alpha)={\sf dom}(\alpha), \hspace{1cm} d_1(\alpha)={\sf codom}(\alpha), \hspace{1cm} s_0(c)=1_c, $$
For higher dimensions $(\NNN\CC)_n$ is defined as the set of all sequences of $n$ composable morphisms
 $$ \bullet \xleftarrow{\alpha_1} \bullet \xleftarrow{\alpha_2} \dots \xleftarrow{\alpha_n} \bullet$$
faces are
$$d_i(\alpha_1,\dots,\alpha_n)=
\begin{cases}
(\alpha_2,\dots,\alpha_n), & i=0 \\
(\alpha_1,\dots,\alpha
_{i}\alpha_{i+1},\dots ,\alpha_n),& 1\leq i\leq n-1\\
(\alpha_1,\dots,\alpha_{n-1}), & i=n
\end{cases}
$$
and degeneracies are
$$s_i(\alpha_1,\dots,\alpha_n)=(\alpha_1,\dots, \alpha_{i},  1_{c_i} , \alpha_{i+1} ,\dots,\alpha_n),$$
where $c_i={\sf dom}(\alpha_i)$ for $i\geq 1$ and $c_0={\sf codom}(\alpha_1).$

Fix an abelian category $\AA$ with exact small direct sums (see \cite[Appendix II]{GabrielZisman}) and enough of projectives.  Since $\AA$ has all small direct sums, it also has all small colimits. Consider the category $\AA^\CC$  of functors
$\CC\to \AA.$ The following lemma seems to be well-known but we can't find a good reference.

\begin{Lemma}The category of functors $\AA^\CC$ is abelian, it has enough of projective objects and projective objects are direct summands of direct sums of functors of type $P^{(\CC(c,-))},$ where $P$ is a projective object of $\AA$ and $P^{(X)}=\bigoplus_X P$ for a set $X.$
\end{Lemma}
\begin{proof} It is obvious that $\AA^\CC$ is abelian.
Note that for any $c\in \CC$ there is an adjunction $\AA^\CC \leftrightarrows \AA $ of functors $\MM \mapsto \MM(c)$ and $A\mapsto A^{(\CC(c,-))}.$ In other words, there is a Yoneda-like isomorphism with the same proof as the proof of the Yoneda lemma $$\AA^\CC(A^{(\CC(c,-))},\MM)=\AA(A,\MM(c)).$$ This isomorphism implies that the functor $\AA^\CC(P^{(\CC(c,-))},-)$ is exact for any projective $P\in \AA,$ and hence $P^{(\CC(c,-))}$ is  projective in the category of functors. For any functor $\MM$ and any object $c\in \CC$ we  fix an epimorphism $P_c\epi \MM(c)$ from a projective object $P_c$ and consider the direct sum of adjoint morphisms $\bigoplus_{c\in \CC} P_c^{(\CC(c,-))} \epi \MM.$ It is easy to see that it is an epimorphism. Thus $\AA^\CC$ has enough of projectives. Moreover, if $\MM$ is projective, the epimorphism splits and we obtain that $\MM$ is a direct summand of a direct sum of functors $P_c^{(\CC(c,-))}.$
\end{proof}

The functor of colimit
$$\colim: \AA^\CC\longrightarrow \AA$$ is left adjoint to the diagonal functor. Hence it
is right exact, and we denote by $\colim_n$ its $n$-th left derived functor.

Let $\MM\in \AA^\CC.$ Consider the simplicial object $C_\bullet(\CC,\MM)$ in $\AA$ (see {\cite[Appendix II,3.2]{GabrielZisman}}) such that
$$C_0(\CC,\MM)=\bigoplus_{c\in (\NNN\CC)_0} \MM(c),$$
$$C_n(\CC,\MM)= \bigoplus_{(\alpha_1,\dots,\alpha_n)\in (\NNN\CC)_n}\MM({\sf dom}(\alpha_n)).$$
The boundary map $d_i:C_n(\CC,\MM)\to C_{n-1}(\CC,\MM)$ is defined so that the restriction on the summand $\MM({\sf dom}(\alpha_n))$ with the index $(\alpha_1,\dots,\alpha_n)$ is just the embedding of the summand with the index $d_i((\alpha_1,\dots,\alpha_n))$ for $i\leq n-1$ and the restriction of  $d_n$ on the same summand is the map
$$\MM(\alpha_n): \MM({\sf dom}(\alpha_n)) \longrightarrow  \MM({\sf codom}(\alpha_{n})) $$
composed with the embedding of the summand with the index $d_n((\alpha_1,\dots,\alpha_n)).$ Degeneracy maps $s_i$ are defined so that  the restriction on the summand with the index $\alpha$ is the embedding to the summand with the index $s_i(\alpha).$  As usual, we  define $\delta_n=\sum (-1)^i d_i :C_n(\CC,\MM)\to C_{n-1}(\CC,\MM)$ and treat $C_\bullet(\CC,\MM)$ as a complex.
  This complex computes the derived colimits:
\begin{Proposition}[{\cite[Appendix II, Proposition 3.3]{GabrielZisman}}]\label{prop_iso_colim_homol} Let $\CC$ be a small category and $\MM:\CC\to \AA$ be a functor. Then there is a natural isomorphism
$$\colim_n\: \MM\cong H_n(C_\bullet(\CC,\MM)).$$
\end{Proposition}

This isomorphism is natural in the following sense. Any functor $\Phi:\CC\to \DD$ between small categories defines a morphism of complexes
$$ C_\bullet(\varphi): C_\bullet(\CC,\MM\Phi ) \longrightarrow C_\bullet(\DD,\MM) $$
that sends the summand with the index $(\alpha_1,\dots,\alpha_n)$ to the summand with the index $(\Phi(\alpha_1),\dots,\Phi(\alpha_n)).$ This morphism of complexes induces a morphism on homology, whose composition with the isomorphisms coincide with the natural map
$$ \colim_n\: \MM\Phi \longrightarrow
 \colim_n\: \MM. $$

A simplicial set $X$ is called contractible, if $ |X|$ is contractible. It is also equivalent to the fact that $X\to *$ is a weak equivalence. A category $\CC$ is said to be contractible if its nerve is contractible. Note that, if there exists a sequence of functors connecting the identity functor and a constant functor ${\sf Id}_\CC=\Phi_1,\dots,\Phi_{2n}=c_0^{\sf const}:\CC\to \CC$ together with natural transformations
$${\sf Id}_\CC=\Phi_0\to \Phi_1 \leftarrow \Phi_2 \to \dots \leftarrow \Phi_{2n}=c_0^{\sf const},$$
then $\CC$ is contractible because these natural transformations induce homotopies between the corresponding maps on the nerve (see \cite[Prop. 2]{QuillenK}). 

\begin{Proposition}\label{prop_coproduct}
Assume that there exists an object  $c_0\in \CC$  such that for any object $c\in \CC$ the coproduct $c\sqcup c_0$ exists. Then $\CC$ is contractible.
\end{Proposition}
\begin{proof}
Consider the functor $\Phi:\mathcal{C}\to
\mathcal{C}$ given by the formula $\Phi(-)=-\sqcup c_0.$ The maps $c\to c\sqcup c_0 \leftarrow c_0$ give natural transformations from
the identity functor $\mathsf{Id}_\mathcal{C}\to \Phi \leftarrow {c_0}^{\sf const}.$ It follows that $\CC$ is contractible. 
\end{proof}

\begin{Proposition}\label{prop_colim_of_constant}
Let $\CC$ be a contractible category and $A\in \AA$ be an object. Consider a constant functor $A^{\sf const}:\CC\to \AA.$ Then $$\colim_0\: A^{\sf const}=A, \hspace{1cm} \colim_n\:  A^{\sf const}=0 $$
for $n\geq 1.$
\end{Proposition}
\begin{proof}
Denote by ${\sf fAb}$ the category of free abelian groups. Then there is a well defined functor of tensor product $\otimes : {\sf fAb}\times \mathcal A\to \mathcal A$ such that $\mathbb Z^{(I)} \otimes A=A^{(I)}.$ 
Note that $$C_\bullet(\mathcal C,A^{\sf const})=C_\bullet({\sf N} \mathcal C)\otimes A,$$ where 
$C_\bullet({\sf N} \mathcal C)$ is the standard chain complex of the simplicial set ${\sf N}\mathcal C.$ Since ${\sf N}\mathcal C$ is contractible, the map $C_\bullet({\sf N}\mathcal C)\to \ZZ[0]$ is a quasi-isomorphism. Using that the complex consists of free abelian groups, we obtain that the map $C_\bullet({\sf N}\mathcal C)\to \ZZ[0]$ is a homotopy equivalence. Then $C_\bullet({\sf N} \mathcal C)\otimes A \to A[0]$ is a homotopy equivalence. 
\end{proof}

During the paper we use homological notation for a complex $M_\bullet:$
$${\dots} \leftarrow  M_{n-1} \leftarrow M_n \leftarrow M_{n+1} \leftarrow \dots.$$

\begin{Proposition}[Spectral sequence of derived colimits I]\label{prop_spectral_sequence_of_colimits} Let $\CC$ be a contractible category, $\MM_\bullet$ be a complex of functors $\CC\to \AA$ such that $\MM_n=0$ for $n<\!<0.$ Assume that its homology $H_n(\MM_\bullet)$ are isomorphic to constant functors. Then there exists a spectral sequence of homological type $E$ in $\AA$ such that
$$E^1_{n,m}= \colim_m\ \MM_n \Rightarrow H_{n+m}(\MM_\bullet).$$
\end{Proposition}
\begin{proof}
Consider two hyper-homology spectral sequences for the functor $\colim:\AA^\CC\to \AA$ and the complex $\MM_\bullet:$  $${}^{I}E^2_{n,m}=  \colim_n\  H_m(\MM_\bullet) \Rightarrow \colim_{n+m}(\MM),  $$
and
$${}^{II} E^1_{n,m}=\colim_m  \ \MM_n \Rightarrow \colim_{n+m}(\MM).$$
Since $H_m(\MM_\bullet)$ is constant, Proposition \ref{prop_colim_of_constant} implies that $\colim_n\  H_m(\MM_\bullet)=0$ for $n\geq 0,$ $\colim_0\  H_m(\MM_\bullet)=H_m(\MM_\bullet),$ and hence, $ \colim_m(\MM_\bullet) =  H_m(\MM_\bullet).$  Then the second spectral sequence is the spectral sequence that we need.
\end{proof}

\begin{Proposition}[Spectral sequence of derived colimits II]\label{prop_spectral_sequence_of_colimits_II} Let $\MM_\bullet$ be a complex of functors $\CC\to \AA$ such that $\MM_n=0$ for small enough $n.$ Assume that
$\colim_{m}\ \MM_n=0  $ for $m\ne 0$ and any $n.$ Set $\MM_n^0=\colim_0\ \MM_n.$ Then there exists a converging spectral sequence of homological type such that
$$E^2_{n,m}=\colim_n\ H_m(\MM_\bullet) \Rightarrow H_{n+m} (\MM^0_\bullet).$$
\end{Proposition}
\begin{proof}
Consider two hyper-homology spectral sequences for the functor $\colim:\AA^\CC\to \AA$ and the complex $\MM_\bullet:$  $${}^{I}E^2_{n,m}=  \colim_n\  H_m(\MM_\bullet) \Rightarrow \colim_{n+m}(\MM),  $$
and
$${}^{II} E^1_{n,m}=\colim_m  \ \MM_n \Rightarrow \colim_{n+m}(\MM).$$
Since $\colim_{m} \MM_n=0  $ for $m\ne 0$,  $ \colim_m(\MM_\bullet) =  H_m(\MM_\bullet^0).$  Then the first spectral sequence is the spectral sequence that we need.
\end{proof}

For a functor $\Phi:\CC\to \DD$ and an object $d\in \DD$ we denote by $d \!\downarrow\! \Phi$ the comma-category. Its objects are couples $(c,\alpha:d\to \Phi(c)),$ where $c\in \CC$ and $\alpha\in \DD(d,\Phi(c)).$ Its morphisms $f:(c,\alpha)\to (c',\alpha') $ are morphisms  $f\in \CC(c,c')$ such that $\Phi(f)\alpha=\alpha'.$

\begin{Proposition}[{cf. Quillen's Theorem A  \cite{QuillenK}}]\label{prop_preserving_of_colimits} Let $\CC,\DD$ be small categories, $\Phi:\CC\to \DD$ be a functor and $\MM:\DD\to \AA$. Assume that the category $d \!\downarrow\! \Phi$ is  contractable for any $d\in \DD.$ Then
$$ {\colim_n}\: \MM\Phi \cong {\colim_n}\: \MM$$
for any $n.$
\end{Proposition}
\begin{proof}
Any contractible category is not empty and connected.  Since the category $d \!\downarrow\! \Phi$ is not empty and connected, the functor $\Phi$ is final (see \cite[IX.3]{MaclaneCat}), and hence $\colim\: \MM \Phi=\colim\: \MM$ for any $\MM\in \AA^\DD.$

The functor of composition $- \circ \Phi :\AA^\DD\to \AA^\CC$ is exact. Prove that it sends projective objects to $\colim$-acyclic objects. Projective objects of $\AA^\DD$ are direct  summunds of direct sums of functors of the form $ P^{(\DD(d,-))},$ where $P$ is a projective object of $\AA.$ So we only need to show that the the objects of $\AA^\CC$ of the form $P^{(\DD(d,\Phi(-)))}$ are $\colim$-acyclic. Note that an $n$-simplex of $\NNN(d\!\downarrow \! \Phi)$ is the same as an $n$-simplex $(c_0\leftarrow \dots \leftarrow c_n)$ of $\NNN\CC$ together with a morphism $d\to \Phi(c_n).$ Therefore, we obtain an isomorphism
$$ \bigoplus_{(c_0\leftarrow \dots \leftarrow c_n)\in (\NNN\CC)_n} P^{\DD(d,\Phi(-))} = \bigoplus_{(\alpha_0\leftarrow \dots \leftarrow \alpha_n)\in (\NNN(d\downarrow \Phi))_n} P.  $$
It is easy to check that the isomorphism is compatible with face maps, and hence, we obtain
$$C_\bullet(\CC,P^{\DD(d,\Phi(-))})=C_\bullet(d\!\downarrow \! \Phi,P).$$ It follows that
$$\colim_*\ P^{(\DD(d,\Phi(-)))}= \underset{d \downarrow  \Phi}{\colim_*}\ P^{\sf const}.$$
 Since $d\!\downarrow\! \Phi $ is contractible Proposition \ref{prop_colim_of_constant} implies that $P^{(\DD(d,\Phi(-)))}$ is $\colim$-acyclic.

Then the spectral sequence of composition together with the fact that $- \circ \Phi :\AA^\DD\to \AA^\CC$ is exact imply the isomorphism $ {\colim_i}\: \MM \cong {\colim_i}\: \MM\Phi.$
\end{proof}

\begin{Proposition}\label{proposition_subcategory_retract} Let $\CC$ be a full subcategory of a small category $\DD$ such that any object of $\DD$ is a retract of an object of $\CC.$ Then for any functor $\MM:\DD\to \AA$ there is an isomorphism
$$\colim_n\ \MM|_{\CC} \cong \colim_n\ \MM$$
for any $n.$
\end{Proposition}
\begin{proof}
Proposition \ref{prop_preserving_of_colimits} implies that we only need to prove that for any fixed object $d_0\in \DD$ the category $d_0 \downarrow \CC$ is contractible. Since $d_0$ is a retract of an object of $\CC,$ we can fix an object $c_0\in \CC$ together with morphisms $r_0:c_0\to d_0$ and $s_0:d_0\to c_0$ such that $r_0s_0={\sf id}.$ Then for any object $\alpha:d_0\to c$ of $d_0\downarrow \CC$ we obtain a commutative diagram
$$
\begin{tikzcd}
&d_0 \arrow[dl,"s_0"'] \arrow[dr,"\alpha"] &\\
c_0\arrow[rr,"\alpha r_0"] && c
\end{tikzcd}
$$
which is natural by $\alpha.$ Therefore, we constructed a natural transformation $(d_0)^{\sf const}\to {\sf Id}_{d_0\downarrow \CC },$ and hence, $d_0\downarrow \CC$ is contractible.
\end{proof}

\begin{Proposition}\label{prop_double_colimits} Let $\CC$ be a small category with pairwise coproducts, $\Phi:\CC\times \CC \to \AA$ be a functor and $\Delta:\CC\to \CC\times \CC$ be the diagonal. Then
$${\colim_n}\ \Phi \Delta\cong {\colim_n}\ \Phi$$
for any $n.$
\end{Proposition}
\begin{proof}
The object $(c_1,c_2)\to (c_1\sqcup c_2,c_1\sqcup c_2)$ is the initial object of the category $(c_1,c_2)\downarrow \Delta.$ Then the category $(c_1,c_2)\downarrow \Delta$ is nonempty and contractible. Proposition  \ref{prop_preserving_of_colimits} implies the assertion.
\end{proof}

\begin{Proposition}[K\"unneth theorem for derived colimits] \label{prop_Kunneth_colim} Let $\CC$ be a small category with pairwise coproducts, $R$ be a principal ideal domain and $\Phi,\Psi:\CC\to {\sf Mod}(R)$ be functors. Assume that $\Psi(c)$ is flat over $R$ for any $c.$  Denote by  $\Phi\otimes_R\Psi $ the functor $\CC\to {\sf Mod}(R)$ that sends $c$ to $\Phi(c)\otimes_R \Psi(c).$ Then there is a short exact sequence of $R$-modules
$$ \bigoplus_{i+j=n} (\colim_i\: \Phi) \otimes_R (\colim_j \Psi) \mono \colim_n\ (\Phi \otimes_R \Psi) \epi
\bigoplus_{i+j=n-1} {\sf Tor}_1^R(\colim_i\: \Phi, \colim_j\: \Psi).$$
\end{Proposition}
\begin{proof} If $A_\bullet$ is a simplicial $R$-module, we denote by $CA_\bullet$ the corresponding complex with the differential $\delta_n=\sum (-1)^id_i$ and  set $H_n(A_\bullet)=H_n(CA_\bullet).$
Recall that if $A_{\bullet}, B_\bullet$ are two simplicial $R$-modules, and $A_{\bullet}\otimes_R^{\sf simpl} B_\bullet$ is their tensor product which is defined level-wise $(A_{\bullet}\otimes_R^{\sf simpl} B_\bullet)_n = A_{n}\otimes_R B_n,$ then there is a natural isomorphism
$$ H_*(A_{\bullet}\otimes_R^{\sf simpl} B_\bullet)\cong  H_*(CA_{\bullet}\otimes_R CB_\bullet), $$
where $CA_{\bullet}\otimes_R CB_\bullet$ is the total tensor product (see \cite[8.5.3]{Weibel}). The K\"uneth theorem for complexes \cite[3.6.3]{Weibel} implies that, if $B_n$ is $R$-flat for any $n,$ then  there is a short exact sequence
$$\bigoplus_{i+j=n} H_i(A_\bullet)\otimes_R H_j(B_\bullet) \mono H_n(A_{\bullet}\otimes_R^{\sf simpl} B_\bullet)\epi \bigoplus_{i+j=n-1} {\sf Tor}_1^R( H_i(A_\bullet), H_j(B_\bullet)).$$
Then we only need to note that
$$C_\bullet(\CC,\Phi)\otimes^{\sf simpl}_R C_\bullet(\CC,\Psi)=C_\bullet(\CC\times \CC,\Phi\tilde \otimes_R\Psi),$$
where $\Phi\tilde \otimes_R\Psi $ is the functor $\CC\times \CC\to {\sf Mod}(R)$ that sends $(c_1,c_2)$ to $ \Phi(c_1) \otimes_R\Psi(c_2),$ and combine this with the identity $\Phi\otimes_R\Psi=(\Phi\tilde \otimes_R\Psi)\Delta$ and Proposition \ref{prop_double_colimits}.
\end{proof}

\begin{Definition}[Strongly connected] A category $\CC$ is called strongly connected, if $\CC(c,c')\ne \emptyset$ (and $\CC(c',c)\ne \emptyset$) for any two objects $c,c'\in \CC$ .
\end{Definition}

\begin{Proposition}\label{colimitintuitive} If $\CC$ is a strongly connected small category, then for any $\MM:\CC\to \AA$ the map
$$i_{c_0}:\MM(c_0)\epi \colim\: \MM$$ is an epimorphism for any $c_0\in \CC$. Moreover, $\colim\: \MM$ is the largest constant quotient of $\MM.$ To be more precise, if we fix $c_0,$ for any two morphisms $\alpha,\beta:c\to c_0$ we have $i_{c_0}\MM(\alpha)=i_{c_0}\MM(\beta)$ and $i_{c_0}$ is initial among those. (Roughly speaking, $\colim\: \MM$ is the coequaliser of all morphisms $\MM(\alpha):\MM(c)\to \MM(c_0)$.)
\end{Proposition}
\begin{proof}
Consider a subcategory $\CC_0$ of $\CC$ with the same objects but with morphisms only to $c_0$ i.e. ${\sf ob}(\CC_0)={\sf ob}(\CC),$ $\CC_0(c,c')=\emptyset,$ if $c'\ne c_0,$ and $\CC_0(c,c_0)=\CC(c,c_0).$ Denote by $\MM_0:\CC_0\to \AA$ the restriction of $\MM.$ Then the assertion can be reformulated as follows
$$\colim\ \MM=\colim\ \MM_0.$$
In order to prove this isomorphism, it is enough to prove that the embedding $\CC_0\to \CC$ is a final functor (see \cite[IX.3]{MaclaneCat}). Fix an object $c_1\in \CC$ and and prove that the category $c_1\! \downarrow\! \CC_0$ is not empty and connected. It is not empty because it contains the object ${\sf id}_{c_1}:c_1\to c_1.$ It is connected because for any its object $\alpha:c_1\to c$ we can chose $\beta\in \CC(c, c_0)$ and construct two morphisms in $c_1\! \downarrow\! \CC_0 $
$$ \begin{tikzcd}
 & c_1\arrow[dl,"{\sf id}"'] \arrow[d,"\beta\alpha"] \arrow[dr,"\alpha"] & \\
c_1 \arrow[r,"\beta\alpha"'] & c_0 & c \arrow[l,"\beta"]
 \end{tikzcd} $$
that connect the object $\alpha$ with the fixed object ${\sf id}_{c_1}.$
\end{proof}

\section{\bf Simplicial resolutions and derived colimits over the category of presentations}

Let $\CC$ be a small category with finite limits and colimits. Recall that a morphism $\alpha:c\to c'$ is called effective epimorphism, if it is a coequalizer of two arrows from the corresponding pullback of $\alpha$ with itself:
$$ \alpha={\rm coeq}( c\times_{c'} c \rightrightarrows c ).$$
An object $p$ of $\CC$ is called projective, if the hom-functor $\CC(p,-)$ sends effective epimorphisms to surjections. A {\it presentation} of an object $c\in \CC$ is an effective epimorphism $p\epi c$ from a projective object.  The category $\CC$ is said to have  sufficiently many projectives if any object has a presentation. Under this assumption a morphism $\alpha:c\to c'$ is an effective epimorphism if and only if $\CC(p,\alpha):\CC(p,c)\to \CC(p,c')$ is surjective for any projective $p$ \cite[II \S 4 Prop 2.]{QuillenH}.

{\bf Assumptions.} Further in the section we assume that $\CC$ is a small category with finite limits and colimits and with enough of projectives. Moreover, we denote by $\AA$ an abelian category with exact small direct sums  and enough of projectives. The main examples for us are $\AA={\sf Mod}(\Lambda)$ and $\AA={\sf Mod}(\Lambda)^{\sf op}$  for some ring $\Lambda$ (direct summs and projectives in ${\sf Mod}(\Lambda)^{\sf op} $ are direct products and injectives in ${\sf Mod}(\Lambda)$).

\begin{Remark}[About set theoretical assumptions]  Here we just assume that $\CC$ is small. Further we will need to use the Tarski-Grothendieck set theory. We will fix two universes $U\in U'$ and consider the category $\CC$ of $U$-small algebraic objects of some kind. For example, the category of $U$-small groups.  Then $\CC$ will be $U'$-small, and we need to consider an abelian category $\AA$ with exact $U'$-small direct sums. For example, the category of $U'$-small modules over some ring.
\end{Remark}

For an object $c\in \CC$ we consider three categories: the first $\CC\!\downarrow\! c$ is the category of objects over $c;$ the second is its full subcategory ${\sf Proj} \!\downarrow\! c$ of projective objects over $c;$ and the most important is the category presentations ${\sf Pres}(c).$ The category ${\sf Pres}(c)$ is the full subcategory of ${\sf Proj} \downarrow c$ whose objects are presentations $p\epi c.$ 
One of advantages of the category ${\sf Pres}(c)$ is that it is strongly connected:

\begin{Lemma} The category ${\sf Pres}(c)$ is strongly connected and has pairwise coproducts. In particular, by Proposition \ref{prop_coproduct} it is contractible.
\end{Lemma}
\begin{proof}
Let $\sigma:p \epi c$ and $\tau : q\epi c$ be two objects of ${\sf Pres}(c).$  Since $\tau$ is an effective epimorphism, the map $\CC(p, \tau): \CC(p, q)\to \CC(p,c)$ is surjective. A preimage $\alpha\in \CC(p, q)$ of $\sigma$ defines a morphism from $\sigma$ to $\tau$ in the category ${\sf Pres}(c).$ Then it is strongly connected. Moreover, note that $p\sqcup q$ is projective as well, and the map $(\sigma,\tau):p\sqcup q \epi c$ is the coproduct in this category.
\end{proof}

\begin{Definition}[Simplicial projective resolution]\label{def_simpl_proj_res} We define a {\it simplicial projective resolution} of an object $c$ as a simplicial object $p_\bullet \in \CC^{\Delta^{\rm op}},$ whose components $p_n$ are projective,  together with a morphism to the constant simplicial object $\varepsilon: p_\bullet\to c^{\sf const}$ such that:
\begin{enumerate}
\item for any projective object $p'$  the morphism $\CC(p',\varepsilon_\bullet):\CC(p',p_\bullet)\to \CC(p',c^{\rm const}) $ is an trivial Kan fibration of simplicial sets,
\item $\varepsilon_0:p_0\to c $ is a presentation.
\end{enumerate}
\end{Definition}

Note that if $\varepsilon_\bullet : p_\bullet \to c^{\sf const}$ is a simplicial projective resolution, then $\varepsilon_n= \varepsilon_0d_0^n,$ $d_0^n:p_n\to p_0$ is a split epimorphism, and hence  $\varepsilon_n:p_n\to c$ is a presentation as well. Therefore, the simplicial projective resolution defines a simplicial object in the category of presentations that we denote by
$$ \varepsilon_\bullet^{\sf P} \in {\sf Pres}(c)^{\Delta^{\rm op}}.$$

As usual, we threat a simplicial object $a_\bullet$ of an abelian category $\mathcal A$  as a complex  with the differential $\delta_n=\sum (-1)^i d_i$  and denote
$$\pi_n(a_\bullet)=H_n(a_\bullet).$$ 

For a morphism $\alpha:c\to c'$ we denote by 
$$\tilde \alpha : {\sf Proj}\! \downarrow\! c \longrightarrow {\sf Proj}\! \downarrow\! c'$$
the corresponding functor of composition with $\alpha$. 

\begin{Proposition}\label{prop_quillen_derived} Let  $ \varepsilon_\bullet : p_\bullet \to c^{\sf const}$ be a simplicial projective resolution of an object $c,$ $\varepsilon^{\sf P}_\bullet$ be the corresponding simpicial presentation and
$$\MM:{\sf Proj}\! \downarrow\! c \to \mathcal A$$ be a functor. Then for any $n$  there is an isomorphism
$$\underset{{\sf Proj} \downarrow c}{\colim_n}\: \MM= \pi_n(\MM(\varepsilon^{\sf P}_\bullet)).$$
\end{Proposition}
\begin{proof} 
The idea of the proof is the following. We construct a double complex $D_{\bullet\bullet}$ whose
 vetical homology satisfy the following
\begin{equation}\label{eq_vert}
H^{\sf vert}_{\bullet,m}(D_{\bullet\bullet})=
\begin{cases} \MM(\varepsilon^{\sf P}_\bullet), & m=0\\
0, & m\ne 0
\end{cases}
\end{equation}
and the horisontal homology satisfy the following
\begin{equation}\label{eq_hor}
H^{\sf hor}_{n,\bullet}(D_{\bullet\bullet})=
\begin{cases} C_\bullet({\sf Proj} \downarrow c, \MM),& n=0\\
0, & n\ne 0.
\end{cases}
\end{equation}
In this case the two spectral sequences of the double complex $D_{\bullet\bullet}$ imply that
$$\colim_*\: \MM = H_*(C_\bullet({\sf Proj} \downarrow c, \MM))=H_*(\MM(\varepsilon^{\sf P}_\bullet)).$$

Then we only need to construct a double complex $D_{\bullet\bullet}$ satisfying \eqref{eq_vert} and \eqref{eq_hor}.

Set
$$D_{nm}=C_m({\sf Proj} \downarrow p_n,\MM \tilde \varepsilon_n). $$ The vertical differentials come from $C_\bullet({\sf Proj} \downarrow  p_n,\MM \tilde \varepsilon_n).$ Hence $D_{n\bullet}=C_\bullet({\sf Proj} \downarrow  p_n,\MM \tilde \varepsilon_n).$ The horizontal differentials $\delta^{hor}_{n,m}:D_{n,m}\to D_{n-1,m}$ are defined as
$$ \delta^{hor}_{n,m}=\sum (-1)^i C_m(\tilde d_i): C_m({\sf Proj} \downarrow  p_n,\MM \tilde \varepsilon_n) \to C_m({\sf Proj} \downarrow  p_{n-1},\MM \tilde \varepsilon_{n-1}).$$

Prove \eqref{eq_vert}. Since the category ${\sf Proj} \downarrow p_n$ has a terminal object ${\sf id}:p_n \to p_n,$ we obtain $\colim \:\NN=\NN({\sf id})$ for any functor $\NN:{\sf Proj} \downarrow p_n\to \AA.$ Therefore, $\colim : {\AA}^{{\sf Proj} \downarrow p_n} \to {\AA}$ is exact and $\colim_n\: \NN=0$ for $n\geq 1.$ It follows that
$$H_m(C_\bullet({\sf Proj} \downarrow p_n,\MM \tilde \varepsilon_n))=
\begin{cases}
\MM\tilde \varepsilon_n({\sf id})=\MM(\varepsilon_n),& m=0\\
0, & m\ne 0.
\end{cases}
$$
Therefore \eqref{eq_vert} is satisfied.

Prove \eqref{eq_hor}. Note that, if $c'$ an object of $\CC$, then an $n$-simplex of $\NNN({\sf Proj} \downarrow  c')$ is a sequence of morphisms of projective objects $$p^{(0)} \xleftarrow{\alpha_1} p^{(1)} \xleftarrow{\alpha_2} \dots \xleftarrow{\alpha_m} p^{(m)}$$ together with morphisms $p^{(i)}\to c'$ such that the diagram is comutative. In order to fix the simplex it is enough to remember the sequence $p^{(0)}\leftarrow p^{(1)} \leftarrow \dots \leftarrow p^{(m)}$ and  the morphism $p^{(0)}\to c'$ because all other morphisms are compositions of these. So we have:
$$C_m({\sf Proj} \downarrow  c',\NN)=\bigoplus_{(p^{(0)}\leftarrow p^{(1)} \leftarrow \dots \leftarrow p^{(m)})} \ \ \bigoplus_{\varphi:p^{(0)}\to c'} \NN(\varphi \alpha_1\dots \alpha_m).$$
Assume that we have a fixed morphism $\varepsilon':c'\to c$  and $\NN=\MM\tilde \varepsilon'.$  Then $\NN(\varphi \alpha_1\dots \alpha_m)=\MM(\varepsilon'\varphi \alpha_1\dots \alpha_m).$ So, it depends only on the composition $\psi:=\varepsilon'\varphi :p^{(0)}\to c.$ Therefore
$$C_m({\sf Proj} \downarrow  c',\MM\tilde \varepsilon')=\bigoplus_{(p^{(0)}\leftarrow  \dots \leftarrow p^{(m)})} \ \ \bigoplus_{\psi:p^{(0)}\to c} \ \ \bigoplus_{\varphi:\psi \to \varepsilon'} \MM(\psi \alpha_1\dots \alpha_m),$$
where the last sum runs over the hom-set $({{\sf Proj} \downarrow  c})( \psi,\varepsilon').$ Since the summand $\MM(\psi \alpha_1\dots \alpha_m)$ does not really depend on $\varphi,$  we can rewrite this in the following form
$$
C_m({\sf Proj} \downarrow  c',\MM\tilde \varepsilon')=\bigoplus_{(p^{(0)}\leftarrow  \dots \leftarrow p^{(m)})} \ \ \bigoplus_{\psi:p^{(0)}\to c} \ \ \MM(\psi \alpha_1\dots \alpha_m)^{\big( ({{\sf Proj} \downarrow  c})( \psi,\varepsilon') \big)  }
.$$
This isomorphism is natural by $ \varepsilon'$. Therefore, in order to prove \eqref{eq_hor} it is sufficient to prove
$$H_n \left(A^{\big( ({{\sf Proj} \downarrow  c})( \psi,\varepsilon_\bullet)\big)} \right)=\begin{cases} A^{\big( ({{\sf Proj} \downarrow  c})( \psi,id_c)\big)}=A,& n=0 \\ 0& n\ne 0 \end{cases}, $$
where $A$ is an object of $\AA$ and $ \psi:p^{(0)}\to c$ is a morphism from a projective object $p^{(0)}.$

 According to the definition given in \cite[Appendix II \S 4]{GabrielZisman} we see that the left-hand homology is the homology of the simplicial set $({{\sf Proj} \downarrow  c})( \psi,\varepsilon_\bullet) $ with coefficients in the object $A.$ So we need to prove that
$$H_n(({{\sf Proj} \downarrow  c})( \psi,\varepsilon_\bullet),A )= \begin{cases} A,& n=0 \\ 0& n\ne 0 \end{cases}.$$ Lemma 4.3 of  \cite[Appendix II \S 4]{GabrielZisman} implies that it enough to prove that the simplicial set $({{\sf Proj} \downarrow  c})( \psi,\varepsilon_\bullet)$ is contractible.

Prove that $({{\sf Proj} \downarrow  c})( \psi,\varepsilon_\bullet)$ is contractible.  Note that $({{\sf Proj} \downarrow  c})( \psi,\varepsilon_n)$ is the preimage of $\psi \in \CC(p^{(0)},c)$ under the map $\CC(p^{(0)},p_n)\to \CC(p^{(0)},c).$ It follows that
$({{\sf Proj} \downarrow  c})( \psi,\varepsilon_\bullet)$ is the fiber of the trivial Kan fibration $\CC(p^{(0)},p_\bullet)\to \CC(p^{(0)},c^{\rm const}).$ Hence, $({{\sf Proj} \downarrow  c})( \psi,\varepsilon_\bullet)$ is fibrant and acyclic. Thus it is contractible.
\end{proof}

\begin{Remark} The proof of Proposition \ref{prop_quillen_derived} does not work for the case of the category ${\sf Pres}(c)$ because in this case $\varphi$ runs over the set of effective epimorphisms 
$\varphi:\psi \epi \varepsilon'$ but not over the whole hom-set.
\end{Remark}

\begin{Proposition}\label{prop_pres_fgr}
Let  $c\in \CC$ be an object,
$\MM:{\sf Proj} \downarrow c\to \AA$ be a functor and $\MM':{\sf Pres}(c)\to \AA$ be its restriction. Then there are isomorphisms
$$ {\colim_n}\: \MM'={\colim_n}\: \MM.$$
\end{Proposition}
\begin{proof} We only need to check that the embedding ${\sf Pres}(c) \hookrightarrow {\sf Proj}\! \downarrow\! c$ satisfies the assumption of Proposition \ref{prop_preserving_of_colimits}. So we need to prove that for any fixed morphism $\varphi_0: p_0\to c$ from a projective object $p_0$ the category $\varphi_0 \!\downarrow\! {\sf Pres}(c)$ is contractible.  An object of the category $\varphi_0 \!\downarrow\! {\sf Pres}(c)$ is a triple $C=(p,\pi,\alpha),$ where $p$ is projective, $\pi:p\epi c$ is a presentation and $\alpha:p_0\to p$ is a morphism such that the diagram
$$ C: \hspace{1cm}
\begin{tikzcd}
p_0\arrow[rr,"\alpha"]\arrow[dr,"\varphi_0"'] & & p\arrow[dl,twoheadrightarrow,"\pi"]\\
& c &
\end{tikzcd}
 $$
is commutative. Chose a fixed presentation $\varepsilon: p_1\epi c.$ Consider the object $C_0=(p_0\sqcup p_1,(\varphi_0,\varepsilon),i_{p_0})$ of the category,  $i_{p_0}:p_0\to p_0\sqcup p_1$ is the standard embedding and $(\varphi_0,\varepsilon):p_0\sqcup p_1\epi c$ is the morphism whose restriction on $p_0$ is $\varphi_0$ and the restriction on $p_1$ is the standard projection $\varepsilon:p_1\epi c$.
$$ C_0: \hspace{1cm}
\begin{tikzcd}
p_0\arrow[rr,"i_{p_0}"]\arrow[dr,"\varphi_0"'] & & p_0\sqcup p_1
\arrow[dl,twoheadrightarrow,"{(\varphi_0,\varepsilon)}"]\\
& c &
\end{tikzcd}
 $$
By Proposition \ref{prop_coproduct} we obtain that it is sufficient to prove that the coproduct $C\sqcup C_0$ exists for any object $C=(p,\pi,\alpha).$ We construct it as follows. Consider the object $C'=(p \sqcup p_1 ,(\pi,\varepsilon),i_{p} \circ \alpha )$
$$ C': \hspace{1cm}
\begin{tikzcd}
p_0\arrow[rr," i_{p} \circ\alpha"]\arrow[dr,"\varphi_0"'] & &  p \sqcup p_1
\arrow[dl,twoheadrightarrow,"{(\pi,\varepsilon)}"]\\
& c &
\end{tikzcd}
 $$
and prove that $C'=C\sqcup C_0.$ Note that the big square
$$
\begin{tikzcd}
p_0\arrow[rr]\arrow[dd]\arrow[dr,"\varphi_0"] & & p_0 \sqcup p_1
\arrow[dl,twoheadrightarrow]\arrow[dd]\\
& c & \\
p\arrow[ur,twoheadrightarrow]\arrow[rr] & & p \sqcup p_1 \arrow[ul,twoheadrightarrow]
\end{tikzcd}
 $$
is a pushout in the category $\CC$. Moreover, one can check that this is a pushout in the category ${\sf Proj}\!\downarrow\! c.$ Generally, if $\DD$ is a category and $d\in \DD,$ then the coproduct in $d\!\downarrow\!\DD$ is the pushout over $d.$ Therefore, the square is a coproduct in $\varphi_0 \downarrow ({\sf Proj} \! \downarrow \! c).$ Since it lies in $\varphi_0\! \downarrow \! {\sf Pres}(c)$, it is a copoduct in $\varphi_0 \downarrow {\sf Pres}(c)$ as well.
\end{proof}

\begin{Theorem}[{cf. \cite[Prop. 5.5.8.15]{Lurie}}]\label{th_derived}
Let  $\varepsilon_\bullet:p_\bullet \to c^{\sf const}$ be a simplicial projective resolution, $\varepsilon^{\sf P}_\bullet$ be the corresponding simpicial presentation,
$\MM:{\sf Proj}\! \downarrow\! c \to \AA$ be a functor and $\MM':{\sf Pres}(c)\to \AA$ be its restriction on the category of presentations. Then for any $n$  there is an isomorphism
$${\colim_n}\: \MM'= \pi_n(\MM(\varepsilon^{\sf P}_\bullet)).$$
\end{Theorem}
\begin{proof}
This follows from Proposition \ref{prop_quillen_derived} and Proposition \ref{prop_pres_fgr}.
\end{proof}

\begin{Corollary}\label{cor_derived} Let $\Phi: \CC\to \AA$ be a functor, $\varepsilon_\bullet:p_\bullet \to c^{\sf const}$ be a siplicial projective resolution. Then for any $n\geq 0$ there is an isomorphism
 $$ \colim_n \Phi'= \pi_n(\Phi( p_\bullet) ) ,$$
where $\Phi':{\sf Pres}(c)\to \AA$ is the functor such that $$\Phi'(p\epi c)=\Phi(p).$$
\end{Corollary}

\begin{Remark} Corollary \ref{cor_derived} implies that the homotopy groups $\pi_n(\Phi( p_\bullet) )$ in our general setting do not depend on the choice of the resolution $p_\bullet.$ Thus    $\colim_n \Phi(p)$ can be considered as a version of a derived functor for a non-additive functor.
\end{Remark}

\begin{Remark}\label{rem_free} Assume that in the category $\CC$ there is a class of ``free'' objects $\mathcal F$  which consists of projective objects, and any projective object is a retract of a on object from $\mathcal F$. Denote by ${\sf Pres}^\FF(c)$ the full subcategory of ${\sf Pres}(c)$
$${\sf Pres}^\FF(c)\subseteq {\sf Pres}(c) $$
consisting of presentations $f\epi c,$ where $f\in \FF.$ For any functor $\MM:{\sf Pres}(c)\to \AA$ we denote by $\MM^\FF:{\sf Pres}^\FF(c)\to \AA$ its restriction.  Then Proposition \ref{proposition_subcategory_retract} implies that  there is an isomorphism
$$ \colim_* \MM=\colim_*\MM^\FF.$$
 \end{Remark}

\begin{Lemma} Let $\CC$ be a category  together with a couple of adjoint functors 
$$F:{\sf Sets} \rightleftarrows \CC :U$$
such that $U$ is a composition of some functor to the category of groups and the forgetful functor from groups to sets:
$$U : \CC \xrightarrow{U'} {\sf Gr} \to {\sf Sets}.$$
Moreover, we assume that a morphism $\alpha$ of $\CC$ is an effective epimorphism if and only if  $U\alpha$ is a surjection.  Then
\begin{itemize}
\item $\CC$ has enough of projectives;

\item an object is projective if and only if it is a retract of $FX$ for some set $X;$

\item  a simplicial object together with a morphism 
$\varepsilon :p_\bullet\to c^{\sf const}$ 
is a simplicial projective resolution if and only if $p_n$ is projective for any $n$ and
 $U'\varepsilon :U'p_\bullet\to U'c^{\sf const}$ is a weak equivalence in ${\sf Gr}.$
\end{itemize} 
\end{Lemma}
\begin{proof}
First note that, if we combine the fact that $U$ sends effective epimorphisms to surjections together with the isomorphism $\CC(FX,-)\cong {\sf Sets}(X,U(-)),$ we obtain that $FX$ is projective for any set $X$. A retract of a projective object is projective. Hence, a retract of $FX$ is projective.  On the other hand, since the composition $Uc\to UFUc\to Uc$ is the identity, we obtain that $UFUc\to Uc$ is a surjection, and hence, $FUc\epi c$ is an effective epimorphism. 
 It follows that $\CC$ has enough of projectives. Moreover, a projective object $p$ is a retract of $FUp.$ Hence an object is projective if and only if it is a retract of $FX$ for some $X.$ 
 
A retract of Kan fibration is a Kan fibration. Thus $\varepsilon:p_\bullet\to c$ is a simplicial projective resolution iff $\CC(FX,p_\bullet) \to \CC(FX,c)$ is a trivial Kan fibration and $\CC(FX,p_0)\to  \CC(FX,c)$ is surjective for any set $X.$ Using the adjunction we obtain that it is equivalent to the fact that $(U'p_\bullet)^X \to (U'c)^X$ is a trivial Kan fibration and $(U'p_0)^X\to (U'c)^X$ is an epimorphism. Since an epimorphism of simplicial groups is always a Kan fibration, then this is equivalent to the fact that $U'p_\bullet \to U'c^{\sf const}$ is a weak equivalence and $U'p_0\to U'c$ is an epimorphism. Since $\pi_0(U'p_\bullet)$ is a quotient of $U'p_0$ and $\pi_0(U'c^{\sf const})=c,$ we obtain that $U'p_0\epi c$ is an epimorphism if $U'p_\bullet\to U'c^{\sf const}$ is a weak equivalence. The assertion follows.
\end{proof}

\begin{Proposition}\label{prop_Andre-Quillen} Let $k$ be a commutative ring, $A$ be a commutative $k$-algebra, $M$ be an $A$-module. Consider the category ${\sf Pres}^{\sf poly}(A)$ of surjective homomorphisms $F\epi A,$ where $F$ is a polynomial $k$-algebra. Then the  
Andr\'e-Quillen homology can be presented as following derived colimit over the category 
 $$D_n(A/k,M)=\underset{{\sf Pres}(A)}{\colim_n}\ \Omega^{\rm comm}(F)\otimes_F M, $$
where $\Omega^{\rm comm}(F)$ is the module of K\"ahler differentials of $F.$ 
\end{Proposition}
\begin{proof}
Effective epimorphisms of the category of commutative $k$-algebras are surjective homomorphisms. Projective objects are retracts of polynomial algebras. Remark \ref{rem_free} show that colimits over ${\sf Pres}(A)$ and ${\sf Pres}^{\sf poly}(A)$ are the same. By the definition we have
$$D_n(A/k,M)=\pi_n\left(\Omega^{\rm comm}(F_\bullet)\otimes_{F_\bullet} M \right)$$
(see \cite[Def. 8.8.2]{Weibel}). Then the assertion follows from Theorem \ref{th_derived}. 
\end{proof}

\section{\bf Groups}

In this section we fix a  group $G$ and all colimits are considered over the category of presentations ${\sf Pres}(G).$ A presentation here is a surjective homomorphism from a free group $F\epi G.$ Its kernel is denoted by $$R={\rm Ker}(F\epi G).$$ We treat $F$ and $R$ as functors ${\sf Pres}(G)\to {\sf Gr}.$

Here we fix two universes $U\in U'$  and assume that $G$ is $U$-small; a presentation $F\epi G$ is $U$-small; denote by ${\sf Gr}$ the category of $U$-small groups; but denote by ${\sf Ab}$ the category of $U'$-small abelian groups.

\begin{Proposition}\label{derpropos} Let $\Phi:{\sf Gr}\to {\sf Ab}$ be a functor and $L_n\Phi$ be the simplicial derived functors of $\Phi$. Then
$$ L_n \Phi(G)=\colim_n \Phi(F),$$
where $\Phi(F):{\sf Pres}(G)\to {\sf Ab}$ is the functor that sends $(F\epi G) $ to $\Phi(F).$
\end{Proposition}
\begin{proof}
This follows from Corollary \ref{cor_derived}.
\end{proof}

\begin{Proposition}\label{prop_Phi(R)} Let $\Phi:{\sf Gr}\to {\sf Ab}$ be a functor. Then
$$ \colim_0\ \Phi(R)=\Phi(1), \hspace{1cm} \colim_n \Phi(R)=0 $$
for $n\geq 1,$
where $\Phi(R):{\sf Pres}(G)\to {\sf Ab}$ is the functor that sends $(F\epi G)$ to $\Phi(R)=\Phi({\rm Ker}(F\epi G)).$
\end{Proposition}
\begin{proof}
The kernel $R_\bullet={\sf Ker}(F_\bullet \epi G^{\sf const})$ is a free resolution of the trivial group. Then  Theorem \ref{th_derived} implies that $\colim_n \Phi(R)=L_n\Phi(1).$
\end{proof}

\begin{Theorem}\label{thm_group_homology} Let $M$ be a $G$-module. Then
$$H_{n+1}(G,M)= \colim_n\ H_1(F,M) $$
for any $n\geq 0.$ In particular,
$$H_{n+1}(G)= \colim_n\ F_{ab}. $$
\end{Theorem}
\begin{proof}
For a group $H$ and a $\ZZ[H]$-module $N$ we denote by $C_\bullet(H,N)$ the standard complex computing homology $H_n(H,N)$ such that $C_n(H,N)=\ZZ[H]^{\otimes n}\otimes N.$ Consider the complex of functors
${\bf P}_\bullet=C_\bullet(F,I(F)\otimes M)$ from ${\sf Proj} \downarrow G $ to ${\sf Ab},$ where $P_m$ is the functor that sends $F\to G$ to $ C_m(F,I(F)\otimes M)$ and $I(F)$ is the augmentation ideal of $\ZZ [F].$ Note that $H_0({\bf P}_\bullet)=H_0(F,I(F)\otimes M)=H_1(F,M).$

 Prove that ${\bf P}_\bullet$ is a $\colim$-acyclic resolution of the functor $H_1(F,M)$ in the category of functors
 ${\sf Proj} \downarrow G\to {\sf Ab}.$ First note that for $n\geq 1$ we have
 $$H_n({\bf P}_\bullet)=H_n(F,I(F)\otimes M)=H_{n+1}(F,M)=0,$$ 
 because $\ZZ[F]$ is hereditary (\cite[Cor. 6.2.7]{Weibel}).
 Then we only need to prove that ${\bf P}_n$ is colim-acyclic.  Take a free simplicial resolution $F_\bullet \overset{\sim}\epi G.$ Then $\pi_n(\ZZ[F_\bullet])=H_n(F_\bullet),$ where $F_\bullet$ is considered as a simplicial set. Since $F_\bullet \epi G^{\sf const}$ is a week equivalence of simplicial sets, this map induces an isomorphism on the level of homology of the simplicial sets.  Then $\pi_n(\ZZ[F_\bullet])=0$ for $n\ne 0$ and $\pi_0(\ZZ[F_\bullet])=\ZZ[G].$  Theorem \ref{th_derived} implies that
 $$ \colim_0\ \ZZ[F]=\ZZ[G], \hspace{1cm} \colim_n\ \ZZ[F]=0  $$
 for $n\ne 0.$ The the long exact sequence of colimits applied to the  short exact sequence $I(F)\mono \ZZ[F]\epi \ZZ$ implies
  $$ \colim_0\ I(F)=I(G), \hspace{1cm} \colim_n\ I(F)=0  $$
 for $n\ne 0.$
Combining this with K\"unneth theorem for colimits  (Proposition \ref{prop_Kunneth_colim}) we obtain
\begin{equation}\label{eq_C_m_in_proof}
 \colim_0\ C_m(F,I(F)\otimes M)=C_m(G,I(G)\otimes M), \hspace{1cm} \colim_n \ C_m(F,I(F)\otimes M)=0
\end{equation}
for $n\ne 0.$ In particular, ${\bf P}_n$ is acyclic and ${\bf P}_\bullet$ is an colim-acyclic resolution of $H_1(F,M).$

Since ${\bf P}_\bullet$ is a colim-acyclic resolution, we can use it for computing of $\colim_n H_1(F,M).$ Equation \eqref{eq_C_m_in_proof} implies that $\colim_0\ {\bf P}_\bullet=C_\bullet(G,I(G)\otimes M).$ Then the assertion follows from the formula $H_n(C_\bullet(G,I(G)\otimes M))=H_n(G,I(G)\otimes M)=H_{n+1}(G,M)$ for $n\geq 0.$
\end{proof}

\begin{Remark} If we take a short exact sequence $M_1\mono M_2\epi M_3$ of $G$-modules and a presentation $F\epi G,$ then we obtain an exact sequence
$$0 \to H_1(F,M_1)\to H_1(F,M_2)\to H_1(F,M_3) \to (M_1)_G\to (M_2)_G \to (M_3)_G \to 0.$$ If we apply the spectral sequence of colimits (Proposition \ref{prop_spectral_sequence_of_colimits}) to this exact sequence and use that $\colim_n\ H_1(F,M_i)=H_{n+1}(G,M_i),$  $\colim_n \ (M_i)_G=0$ for $n\geq 1,$ and $\colim_0 (M_i)_G=(M_i)_G,$ then we obtain the long exact sequence
$$  \dots \to H_n(G,M_1)\to H_n(G,M_2)\to H_n(G,M_3) \to H_{n-1}(G,M_1) \to \dots.$$
\end{Remark}

\section{\bf Hochschild and cyclic homology of unital algebras}

In this section we assume that all algebras are unital and associative. We assume that $k$ is a field and denote by ${\sf Alg}^u$ the category of unital algebras over $k$ and $\otimes=\otimes_k$. Effective epimorphisms in this category are surjective homomorphisms, and projective objects are retracts of free algebras. A free algebra $F$ is isomorphic to the tensor algebra $F\cong T(V)=\bigoplus_{n\geq 0} V^{\otimes n}.$ Then the objects of the category of presentations ${\sf Pres}(A)$ are surjective homomorphisms $F\epi A$ from retracts of free algebras. In this section we consider only colimits of functors of type ${\sf Pres}(A)\to {\sf Vect}.$

Here we fix two universes $U\in U'$  and assume that $A$ is $U$-small; a presentation $F\epi A$ is $U$-small; denote by ${\sf Alg}^u$ the category of $U$-small unital algebras; but denote by ${\sf Vect}$ the category of $U'$-small vector spaces.

Consider the subcategory $${\sf Pres^{free}}(A)\subseteq {\sf Pres}(A)$$ consisting of presentations $F\epi A,$ where $F$ is free. Remark \ref{rem_free} implies the isomorphism $$\underset{{\sf Pres}(A)}{\colim_*}\ \MM \cong \underset{{\sf Pres^{free}}(A)}{\colim_*}\ \MM $$
for any functor $\MM:{\sf Pres}(A)\to {\sf Mod}(k).$ So we can restrict ourselves by considering only such presentations $F\epi A,$ where $F$ is a free algebra, if a functor $\MM$ can be defined for all presentations.  For a presentation $F\epi A$ we set
$R={\rm Ker}(F\epi A).$

For an algebra $A$ and an $A$-bimodule $M$ we set
$$M_\natural:=M/[M,A]=HH_0(A,M).$$

For an algebra $A$ we set
$$A^e=A^{op} \otimes A.$$
If $M$ is an $A$-bimodule, we can consider it both as a left $A^e$-module via $(a\otimes b)*m=bma$ and as a right $A^e$-module via $m*(a\otimes b)=amb.$ Note that with this definition we have isomorphisms
$$ M\otimes_{A^e} N \cong N\otimes_{A^e} M\cong (M\otimes_A N)_\natural \cong (N\otimes_A M)_\natural$$
for any two $A$-bimodules $M,N.$

\subsection{$\OO$-modules and $\OO$-bimodules}

We denote by ${\sf Mod}^{\sf r}$ the category of couples $(M,B),$ where $B\in {\sf Alg^u}$  and $M$ is a right $B$-module. Morphisms in this category are couples $(f,\varphi):(M,B)\to (M',B'),$ where $ \varphi:B\to B'$ is a homomorphism of  algebras and $f:M\to M'$ is a homomorphism of $B$-modules, where $M'$ is considered as an $A$-module via $\varphi.$ There is an obvious forgetful functor
$$ {\sf Mod^r} \longrightarrow {\sf Alg^u}, \hspace{1cm} (M,B) \mapsto B.$$

Consider the forgetful functor
$$ \OO: {\sf Pres}(A) \longrightarrow {\sf Alg^u}, \hspace{1cm} \OO(F\epi A)=F. $$
A right $\OO$-module is a functor $\MM:{\sf Pres}(A) \to {\sf Mod}^{\sf r}$ such that the diagram
$$
\begin{tikzcd}
 & & {\sf Mod^{r}}\arrow[d] \\
{\sf Pres}(A) \arrow[rr,"\OO"] \arrow[urr,"\MM"] & & {\sf Alg^u}
\end{tikzcd}
$$
is commutative. An $\OO$-homomorphism is a natural transformation $\MM\to \MM',$ whose second component consists of identity homomorphisms ${\sf id}_F:F\to F.$ Thus we obtain a category of right $\OO$-modules ${\sf Mod^r}(\OO).$ Similarly one can define the category of left $\OO$-modules ${\sf Mod^l}(\OO)$ and the category of $\OO$-bimodules ${\sf Bimod}(\OO).$

 By abuse of notation, for an $\OO$-module $\MM$ we will identify $\MM( F\epi A)$ with the underlying $F$-module $\MM(F\epi A)=(\MM(F\epi A),F).$ Then $\MM(F\epi A)$ is a vector space together with a structure of $F$-module, which is ``natural by a presentation''. For an $\OO$-module $\MM$ we set
$$ \colim_n\ \MM:=\colim_n ( {\sf Pres}(A)  \xrightarrow{\MM} {\sf Mod^r} \to {\sf Vect}), $$
where  ${\sf Mod^r} \to {\sf Vect}$ is the forgetful functor $(M,B)\mapsto M$.

We claim that $\MM_0:=\colim_0\ \MM$ has a natural structure of $A$-module. Indeed, Theorem \ref{th_derived} implies that $\colim_0\ F=A,$ and the K\"unneth formula for colimits (Proposition \ref{prop_Kunneth_colim}) implies that $\colim_0\ F\otimes \MM(F\epi A)=A\otimes \MM_0$. Therefore, the natural transformation $ F\otimes \MM(F\epi A) \to \MM(F\epi A)$ induces a map $A\otimes \MM_0 \to \MM_0.$ It is easy to check that this map defines a structure of a module on $\MM_0.$ Therefore, we obtain a well-defined functor
$$ {\sf Colim}_0: {\sf Mod^r}(\OO) \longrightarrow {\sf Mod^r}(A),$$
and similarly
$${\sf Colim}_0: {\sf Mod^l}(\OO) \longrightarrow {\sf Mod^l}(A), \hspace{1cm} {\sf Colim}_0: {\sf Bimod}(\OO) \longrightarrow {\sf Bimod}(A). $$

\begin{Proposition} \label{prop_O-bimodules-tor}
Let $\MM$ and $\NN$ be colim-acyclic $\OO$-bimodules. Set $\MM_0={\sf Colim}_0\ \MM$ and $\NN_0={\sf Colim}_0 \ \NN.$ Then there is a spectral sequence of homological type
$$
E^2_{n,m}=\colim_n {\sf Tor}^{F^e}_m(\MM,\NN) \Rightarrow {\sf Tor}^{A^e}_{n+m}(\MM_0,\NN_0).
$$
\end{Proposition}
\begin{proof}
For an algebra $B$ we denote by ${\sf Bar}_\bullet(B)$ the standard bar resolution of the bimodule $B$, where
$ {\sf Bar}_n(B)= B^{\otimes n+2} $ and
$$d(b_0\otimes \dots \otimes b_{n+1})=\sum (-1)^i b_0\otimes \dots \otimes b_{i}b_{i+1} \otimes \dots \otimes b_{n+1}.$$
For a right $B$-module $L$ and a left $B$-module $L'$ we set
$$C^B_\bullet(L,L')=L\otimes_B {\sf Bar}_\bullet(B) \otimes_B L'.$$ It is easy to check that
$$H_n(C^B_\bullet(L,L'))={\sf Tor}^B_n(L,L'),$$
$$C^B_n(L,L')=L\otimes B^{\otimes n}\otimes L'$$
and the differential is given by the same formula as in ${\sf Bar}_\bullet(B).$

We consider the following complex of functors from ${\sf Pres}(A)$ to ${\sf Vect}$
$${\bf C}_\bullet=C^{F^e}_\bullet(\MM,\NN).$$ Note that ${\bf C}_n=\MM \otimes F^{\otimes 2n}\otimes \NN$ and $ H_n({\bf C}_\bullet)={\sf Tor}^{F^e}_n(\MM,\NN).$ The K\"unneth formula for colimits together with isomorphism $\colim_n\ \MM=\colim_n\ \NN=0$ for $n\ne 0$ implies
$$ \colim_n\ {\bf C}_\bullet= 0, \hspace{1cm} \colim_0\ {\bf C}_\bullet=C^{A^e}_\bullet(\MM_0,\NN_0)   $$ for $n\ne 0.$ Then the assertion follows from Proposition \ref{prop_spectral_sequence_of_colimits_II}.
\end{proof}

\begin{Corollary}\label{cor_one-O-bimod}
Let $\MM$ be a colim-acyclic $\OO$-bimodule. Set $\MM_0={\sf Colim}_0\ \MM.$ Then there is a natural by $\MM$ long exact sequence
$$
\begin{tikzcd}
\colim_{n-1} \ H_1(F,\MM)\arrow[r] & H_n(A,\MM_0)\arrow[r] & \colim_n\: \MM_\natural  \arrow[ddll,out=315, in=135]  \\ &  & \\
\colim_{n-2} \ H_1(F,\MM) \arrow[r] &  H_{n-1}(A,\MM_0) \arrow[r] & \colim_{n-1}\: \MM_\natural.
\end{tikzcd}
$$
\end{Corollary}
\begin{proof}
It follows from  Proposition \ref{prop_O-bimodules-tor} if we take one of the $\OO$-bimodules equal to $F$ and use that $H_n(F,-)=0$ for $n\geq 2$ and $H_0(F,-)=(-)_\natural$  \cite[Prop 9.1.6]{Weibel}.
\end{proof}

\begin{Corollary}\label{cor_tensor_product_O_mod}
Let $\MM$ be a right $\OO$-module and  $\NN$ be a left $\OO$-module. Assume that $\MM$ and $\NN$ are colim-acyclic. Set $\MM_0={\sf Colim}_0\ \MM $ and $\NN_0={\sf Colim}_0 \ \NN.$ Then there is a long  exact sequence
$$
\begin{tikzcd}
\colim_{n-1} H_1(F, \NN\otimes \MM) \arrow[r] & H_n(A,\NN_0\otimes\MM_0) \arrow[r] &     \colim_n\ \MM\otimes_F \NN \arrow[ddll,out=315, in=135]  \\ &  & \\
\colim_{n-2} H_1(F, \NN\otimes \MM) \arrow[r] &  H_{n-1}(A,\NN_0\otimes\MM_0) \arrow[r] & \colim_{n-1}\ \MM\otimes_F \NN.
\end{tikzcd}
$$
\end{Corollary}
\begin{proof}
This follows from Corollary \ref{cor_one-O-bimod} together with the isomorphism $(\NN\otimes \MM)_\natural=\MM\otimes_F \NN$ and the K\"unneth formula for colimits.
\end{proof}

\begin{Lemma}\label{lemma_R^n} For any $n$ and $1\leq m\leq  l$ we have
$$ \colim_n\ R^m=0, \hspace{1cm} \colim_n\ R^m/R^l=0.$$
\end{Lemma}
\begin{proof}
The second isomorphism follows from the first one and the short exact sequence $R^l\mono R^m \epi R^m/R^l.$

Prove the first equality. The proof is by induction. For $m=1$ it follows from the fact that the map $F\epi A$ induces an isomorphism on the level of colimits $\colim_*\ F=\colim_*\ A$ (which follows from  Theorem \ref{th_derived}). Do the step. First note that for any two ideals $\mathfrak{a},\mathfrak{b}$ of a ring $\Lambda$ there is a short exact sequence ${\sf Tor}_2^\Lambda(\Lambda/\mathfrak{a},\Lambda/\mathfrak{b})\mono \mathfrak{a}\otimes_\Lambda \mathfrak{b} \epi \mathfrak{ab}$ (see Exercise 19 of Chapter VI in  \cite{CartanEilenberg}).  It follows that  $R\otimes_F R^m=R^{m+1}.$ Moreover, $R,R^m$ are projective as $F$-modules (right and left) because $F$ is hereditary \cite[Prop 9.1.6]{Weibel}. It follows that $R^m \otimes R$ is a projective bimodule. Then the assertion follows from Corollary \ref{cor_tensor_product_O_mod}.
\end{proof}

\subsection{Hochschild homology}

Consider the kernel of the multiplication map
$$\Omega(A)={\rm Ker}(A\otimes A\epi A).$$ Since the beginning of the bar resolution is exact $A^{\otimes 3} \to A^{\otimes 2} \to A,$ the map $A^{\otimes 3}\to \Omega(A)$ given by $a\otimes b\otimes c\mapsto ab\otimes c - a\otimes bc$ is an epimorphism. Hence, any element of $\Omega(A)$ can be presented as a sum of elements of the form $ab\otimes c - a\otimes bc.$ If we identify $A\otimes A=A^e,$ then $\Omega(A)$ is a right ideal of $A^e$ but it is not necessarily a left ideal.
The bimodule $\Omega(A)$ is known as the bimodule of noncommutative differential forms and the map
$$ d:A\longrightarrow \Omega(A), \hspace{1cm} d(a)=a\otimes 1 - 1\otimes a $$
is the universal derivation. Here we always consider $A\otimes A$ as a bimodule with the following structure $a'(a\otimes b)b'=a'a\otimes bb'.$ If we consider the long exact sequence of ${\sf Tor}^{A^e}_*(M,-)$ applied to the short exact sequence $\Omega(A)\mono A^e \epi A$ of right $A^e$-modules, we obtain a short exact sequence
\begin{equation}\label{eq_h1_def}
0 \longrightarrow H_1(A,M) \longrightarrow \Omega(A)\otimes_{A^e} M \overset{\alpha_M}\longrightarrow M \longrightarrow M_\natural \longrightarrow 0.
\end{equation}
Detailed understanding of the map $\alpha_M$ will be important.
If we denote by  $i_A:\Omega(A)\mono A\otimes A $ the embedding, then $\alpha_M:\Omega(A)\otimes_{A^e}M\to M$ can be written as
$$\alpha_M=i_A\tilde \otimes 1_M,$$ where $i_A\tilde \otimes 1_M$ is the composition of the map $i_A\otimes 1_M : \Omega(A)\otimes_{A^e} M \to A^e \otimes_{A^e}M $ and the isomorphism $A^e\otimes_{A^e}M\cong M$ given by $a\otimes b \otimes m \mapsto bma.$ Therefore, $$\alpha_M((ab\otimes c - a\otimes bc)\otimes m)=cmab-bcma=[cma,b].$$
In particular, we have
$$0 \longrightarrow HH_1(A) \longrightarrow \Omega(A)_\natural \overset{\alpha}\longrightarrow A \longrightarrow A_\natural \longrightarrow 0, $$
such that $\alpha(ab\otimes c - a\otimes bc+[\Omega(A),A])=[ca,b].$

\begin{Theorem}\label{theorem_hochsch} For any $A$-bimodule $M$ there is an isomorphism
$$H_{n+1}(A,M)\cong \colim_n\ H_1(F,M)$$
for $n\geq 0,$ and isomorphism
$$H_{n+1}(A,M)\cong \colim_n\  \Omega(F)\otimes_{F^e} M$$
for $n\geq 1,$
and  isomorphisms
$$ HH_{n+1}(A)\cong \colim_n\ \Omega(F)_\natural  = \colim_n \ \Omega(F)\otimes_{F^e} A $$
for $n\geq 1$. Moreover, the morphisms
$$H_1(F,A)\mono \Omega(F)\otimes_{F^e} A \twoheadleftarrow \Omega(F)_\natural $$ induce  isomorphisms
$$ \colim_n\ H_1(F,A) \cong \colim_n \  \Omega(F)\otimes_{F^e} A \cong  \colim_n\ \Omega(F)_\natural$$
for $n\geq 1,$ which are compatible with the above isomorphisms,  and the following on the level of zero colimits
$$ \colim_0\ H_1(F,A)\cong HH_1(A) \mono \Omega(A)_\natural \cong  \colim_0\ \Omega(F)\otimes_{F^e} A =\colim_0\ \Omega(F)_\natural.$$
\end{Theorem}
\begin{proof}
The first isomorphism follows from Corollary \ref{cor_one-O-bimod} if we take the constant $\OO$-bimodule $\MM=M,$ use that $ M_\natural =M/[M,A] $ is a constant functor and that higher colimits of a constant functor are trivial. 

The second isomorphism follows from the first isomorphism and the short exact sequence with the constant last term $H_1(F,M) \mono \Omega(F)\otimes_{F^e}M \epi [M,F].$ 

The $\OO$-bimodules $F$ and $F\otimes F$ are colim-acyclic. Hence $\Omega(F)$ is a colim-acyclic $\OO$-bimodule as well. The isomorphism $HH_{n+1}(A)\cong HH_{n}(A,\Omega(A))\cong \colim_n\ \Omega(F)_\natural$ for $n\geq 1$ follows from Corollary \ref{cor_one-O-bimod} if we take the $\OO$-bimodule given by $\MM(F\epi A)=\Omega(F)$ and use that $H_1(F,\Omega(F))=0$.

Since $\Omega(F)$ is a free $F^e$-module (see \cite[Prop. 5.8]{QuillenC} or \cite[Rem 3.1.3]{Loday}), we obtain ${\sf Tor}_n^{F^e}(\Omega(F),R)=0 $ for $n\ne 0.$
Then Proposition \ref{prop_O-bimodules-tor} implies that $\colim_n\ \Omega(F)\otimes_{F^e} R={\sf Tor}_n^{A^e}(\Omega(A),0)=0.$  The short exact sequence $R\mono F\epi A$ gives the following short exact sequence
$$0 \longrightarrow \Omega(F)\otimes_{F^e} R \longrightarrow \Omega(F)_\natural \longrightarrow \Omega(F)\otimes_{F^e} A \longrightarrow 0.$$
Combining this short exact sequence with  the isomorphism $\colim_* \Omega(F)\otimes_{F^e} R=0$, we obtain that the epimorphism $\Omega(F)_\natural \epi \Omega(F)\otimes_{F^e}A$ induces an isomorphism $\colim_*\ \Omega(F)_\natural=\colim_*\ \Omega(F)\otimes_{F^e}A.$

Finally, the short exact sequence with the constant last term
$$0 \longrightarrow H_1(F,A) \longrightarrow \Omega(F)\otimes_{F^e}A \longrightarrow [A,A] \longrightarrow 0 $$
implies that the map $H_1(F,A)\mono \Omega(F)\otimes_{F^e} A$ induces an isomorphism $\colim_n \ H_1(F,A)=\colim_n \ \Omega(F)\otimes_{F^e}A$ for $n\ne 0$ and the short exact sequence $\colim_0\ H_1(F,A) \mono \Omega(A)_\natural \epi [A,A].$
\end{proof}

\begin{Proposition}\label{prop_h1(f,omega)} There is a natural isomorphism
$$H_1(F,A^e)=R/R^2 $$
and a natural short exact sequence
$$0 \longrightarrow \frac{R^2+[R,F]}{R^2} \longrightarrow H_1(F,\Omega(A)) \longrightarrow HH_2(A) \longrightarrow 0 $$
such that the diagram
$$\begin{tikzcd}
({R^2+[R,F]})/{R^2} \arrow[r,rightarrowtail] \arrow[d,hookrightarrow] & H_1(F,\Omega(A)) \arrow[r,twoheadrightarrow] \arrow[d,rightarrowtail] & HH_2(A)\\
R/R^2 \arrow[r,"\cong"] & H_1(F,A\otimes A)
\end{tikzcd} $$
is commutative.
\end{Proposition}
\begin{proof}
There is an exact sequence \cite[Prop 7.2]{IvanovMikhailov} of  $A$-bimodules
$$ 0 \longrightarrow R/R^2 \xrightarrow{\delta} A \otimes_F  \Omega(F)\otimes_{F} A \xrightarrow{\ f\ } A\otimes A,$$
where $\delta(r+R^2)=1\otimes d(r)\otimes 1$ and $f$ is an $A$-bimodule homomorphism such that  $f(1\otimes (x\otimes 1 - 1\otimes x)\otimes 1)=\bar x\otimes 1 - 1\otimes \bar x,$ where $x\in F$ and $\bar x$ is its image in $A.$ We call it ``Magnus embedding for algebras''. Consider the isomorphism of $A$-bimodules $A\otimes_F \Omega(F)\otimes_F A\cong  \Omega(F)\otimes_{F^e} A^e$ given by $a\otimes m \otimes b \mapsto  m \otimes (a \otimes b).$ Here we assume that the structure of a left $A^e$-module on $A^e$ is given by multiplication in $A^e=A^{op}\otimes A$ (it is different from $A\otimes A$ and uses the ``inner side'': $(a\otimes b)*(a'\otimes b')=a'a\otimes bb'$). And the structure of the right $A^e$-module is given by the multiplication as well but it is more standard because uses ``exterior'' multiplication  and $A^e=A\otimes A$ as right $A^e$-modules.   Then we can rewrite this exact sequence as follows
\begin{equation}\label{eq_es1}
0 \longrightarrow R/R^2 \xrightarrow{\delta'}   \Omega(F) \otimes_{F^e} A^e \xrightarrow{\ f' \ } A^e
\end{equation}
where $\delta'(r+R^2)=d(r) \otimes   1\otimes 1$ and $f'$ is a right $A^e$-module homomorphism such that $f'(m\otimes 1\otimes 1)=\bar m,$ where $m\in \Omega(F)$ and $\bar m$ is its image in $\Omega(A).$ Note that
$$f'=i_F  \tilde\otimes  1_{A^e}=\alpha_{A^e} ,$$ where $i_F\tilde \otimes 1_{A^e}$ is the composition of $ i_F\otimes 1_{A^e}:A^e \otimes \Omega(F) \to A^e \otimes_{F^e} F^e$ together with the isomorphism $F^e\otimes_{F^e} A^e\cong A^e.$  Indeed, both of them are $A$-bimodule homomorphisms that send $m\otimes 1\otimes 1 $ to $\bar m,$ where $m\in \Omega(F)$ and $\bar m$ is its image in $\Omega(A).$ Therefore,
$$R/R^2={\rm Ker}(f')={\rm Ker}(\alpha_{A^e})=H_1(F,A^e).$$

The image of $f'=\alpha_{A^e}$ is equal to $\Omega(A).$ Then we have the following exact sequence of right $A^e$-modules
\begin{equation}\label{eq_se2}
0 \longrightarrow R/R^2 \xrightarrow{\delta'}   \Omega(F) \otimes_{F^e} A^e \xrightarrow{\ \tilde f' \ } \Omega(A) \longrightarrow 0.
\end{equation}
Consider the map ${\sf tw}:A\otimes A\to A\otimes A$ given by ${\sf tw}(a\otimes b)=b\otimes a$ and set
$$\Omega'(A)={\sf tw}(\Omega(A)).$$ Then $\Omega'(A)$ is a left ideal of $A^e.$ If we consider $\Omega'(A)$ as an $A$-bimodle, then $\Omega'(A)\cong \Omega(A).$
We tensor the short exact sequence \eqref{eq_se2} by $\Omega'(A)$ and use that there is a monomorphism $i'_A:\Omega'(A)\mono A^e$ of left $A^e$-modules:
$$\begin{tikzcd}
 R/R^2  \otimes_{A^e} \Omega'(A)\arrow[d] \arrow[r,"\delta\otimes 1  "] &
 \Omega(F) \otimes_{F^e}\Omega'(A)\arrow[d,rightarrowtail] \arrow[r,twoheadrightarrow,"  pr\otimes 1"]
 \arrow[rd,"pr \tilde \otimes i'_A" description] & \Omega(A)\otimes_{A^e} \Omega'(A) \arrow[d,"1\tilde \otimes i'_A"]
\\
 R/R^2 \arrow[r,rightarrowtail,] & \Omega(F) \otimes_{F^e}A^e\arrow[r,twoheadrightarrow] & \Omega(A)
\end{tikzcd}
$$
Here we use the isomorphism $M\otimes_{A^e}A^e=M$ given by $m\otimes(a\otimes b)\mapsto amb.$ Then there is a short exact sequence
$$ 0 \longrightarrow {\rm Ker}(pr \otimes 1_{\Omega(A)}) \longrightarrow {\rm Ker}(pr \tilde \otimes i'_A) \longrightarrow {\rm Ker}(1_{\Omega(A)} \tilde \otimes i'_{A}) \longrightarrow 0.$$
Prove that this short exact sequence is the sequence that we need.

Since $\Omega(F)$ is a free bimodule, we get that the map $\Omega(F)\otimes_{F^e}\Omega(A) \to   \Omega(F)  \otimes_{F^e}  A^e$ is a monomorphism. It follows that the kernel $pr \otimes  1$ is isomorphic to the image of $R/R^2\otimes_{A^e}  \Omega'(A)\to R/R^2$   which is equal to $[R/R^2, A]=(R^2+[R,F])/R^2:$
$${\rm Ker}(pr\otimes 1_{\Omega(A)})= (R^2+[R,F])/R^2.$$

Note that $pr \tilde \otimes 'i_A =i_F \tilde \otimes 1_{\Omega'(A)}.$ Indeed, both of them are induced by the multiplication in $A^e.$ Therefore
$${\rm Ker}(pr \tilde \otimes i_A)=H_1(F,\Omega'(A))=H_1(F,\Omega(A)).$$
Finally, by the same reason we see that
$${\rm Ker}(1_{\Omega(A)}\otimes i'_A)={\rm Ker}(i_{A}\otimes 1_{\Omega'(A)})=H_1(A,\Omega'(A))=HH_2(A).$$
\end{proof}

\begin{Corollary}\label{cor_R^2+[R,F]} For $n\geq 1$ there is an isomorphism $$HH_{n+2}(A)=\colim_n\ R^2+[R,F]$$
and $\colim_0 \ R^2+[R,F]=0.$
\end{Corollary}
\begin{proof}
Theorem \ref{theorem_hochsch} implies that $\colim_n \ H_1(F,\Omega(A))=H_{n+1}(A,\Omega(A))=HH_{n+2}(A)$ for  $n\geq 0.$ The short exact sequence with the constant functor in the end $(R^2+[R,F])/R^2 \mono H_1(F,\Omega(A))\epi HH_2(A)$ (Proposition \ref{prop_h1(f,omega)}) implies that $\colim_n \ (R^2+[R,F])/R^2 = HH_{n+2}(A)$ for $n\geq 1$ and $\colim_0\ (R^2+[F,F])/R^2=0.$ Then, using $\colim_* R^2=0$ (Lemma \ref{lemma_R^n}) and the short exact sequence $R^2\mono (R^2+[F,F])\epi (R^2+[F,F])/R^2$ we obtain $ \colim_n \ (R^2+[R,F])=\colim_n \ (R^2+[R,F])/R^2.$
\end{proof}

\subsection{Reduced cyclic homology over a field of  characteristic zero}

In this subsection we will always assume that ${\sf char}(k)=0.$

For a unital algebra $B$ we denote by $\bar B$ the quotient $\bar B=B/k\cdot 1.$ Moreover, for a free algebra $F$  we set $$\bar F_\natural = F/(k\cdot 1+[F,F]).$$ It is easy to see that there is a short exact sequence
$$ 0 \longrightarrow k \longrightarrow F_\natural \longrightarrow  \bar F_\natural \longrightarrow 0.$$
As usual, we consider $\bar F_\natural$ as a functor from the category of presentations that sends $F\epi A$ to $\bar F_\natural.$

We denote by $\overline{HC}_n(A)$ the reduced cyclic homology of $A$ \cite[\S 2.2.13]{Loday}.  Note that
$$\bar F_\natural =\overline{HC}_0(F).$$

\begin{Theorem}[{cf. \cite{DonadzeInassaridzeLadra}}]\label{theorem_cyclic} Assume that ${\sf char}(k)=0.$ Then for any $n\geq 0$ there is an isomorphism
$$\colim_n\ \bar F_\natural= \overline{HC}_n(A).$$
\end{Theorem}
\begin{proof}
We follow notation of Loday \cite[\S 2.1.9]{Loday} and denote by  by $\bar \BB(A)$ the double complex that computes the cyclic homology of an algebra $A,$ where
$$\bar \BB(A)_{n,m}= A\otimes \bar A^{\otimes (m-n)}$$
for $m\geq n$ and $\bar \BB(A)_{n,m}=0$ for $m<n.$ Further, following \cite[\S 2.1.13]{Loday}, we denote by $\bar \BB(B)_{\sf red}$ the double complex that computes the reduced cyclic homology
$\bar \BB(A)_{\sf red}=\bar \BB(A)/\bar \BB(k).$
It is easy to see that
$$ (\bar \BB(A)_{\sf red})_{n,m}= \bar A^{\otimes (m-n+1)} $$
for $m\geq n$ and $0$ otherwise. Consider the complex  ${\bf P}_\bullet$ of functors ${\sf Pres}(A)\to {\sf Vect}$ given by
$${\bf P}_\bullet(F\epi A)= {\sf Tot}\ \bar \BB(F)_{\sf red}.$$

Prove that ${\bf P}_\bullet$ is a colim-acyclic resolution of the functor $\bar F_\natural.$ By \cite[Theorem 3.1.6]{Loday} we have that $HC_n(F)=HC_n(k)$ for $n\geq 1.$ Combining this with the fact that $F$ is an augmented algebra, we obtain $\overline{HC}_n(F)=0$ for  $n\geq 1.$ Therefore
$$H_n({\bf P}_\bullet)=\overline{HC}_n(F)=0, \hspace{1cm} H_0({\bf P}_\bullet)=\overline{HC}_0(F)=\bar F_\natural$$
for $n\geq 1$. Now we need to prove that ${\bf P}_n$ is colim-acyclic. Theorem \ref{th_derived} implies that
$$\colim_n\ F=0, \hspace{1cm}  \colim_0 \ F=A $$
for $n\geq 1.$ Using the short exact sequence $k\mono F\epi \bar F$, we obtain
$$\colim_n\ \bar F=0, \hspace{1cm}  \colim_0 \ \bar F=\bar A.$$ Then the K\"unneth theorem for colimits (Proposition \ref{prop_Kunneth_colim}) implies that
$$ \colim_n \ {\bf P}_\bullet =0, \hspace{1cm} \colim_0 \ {\bf P}_\bullet={\sf Tot}\ \bar \BB(A)_{\sf red}$$
for $n\geq 1.$ This follows that ${\bf P}_\bullet$ is a colim-acyclic resolution. Then the assertion follows from the equation $ \colim_0 \ {\bf P}_\bullet={\sf Tot}\ \bar \BB(A)_{\sf red}$.
\end{proof}

\begin{Lemma}\label{lemma_Omega(F)_decomposition} Let ${\sf char}(k)=0$  and $F$ be a free $k$-algebra. Then there is a short exact sequence
$$0 \longrightarrow \bar F_\natural \overset{\tilde d}\longrightarrow \Omega(F)_\natural \overset{\alpha}\longrightarrow [F,F] \longrightarrow 0,$$
where $\tilde{d}$ is induced by the universal derivation $d:F\to \Omega(F),$ $d(a)=a\otimes 1 - 1\otimes a$ and $\alpha(ab\otimes c - a\otimes bc+[\Omega(F),F])=[ca,b].$  In particular, there is an isomorphism
$HH_1(F)=\bar F_\natural.$
\end{Lemma}
\begin{proof}
We denote by $C_n=\langle t \mid t^n=1 \rangle $ the cyclic group generated by $t.$  It is well known that
$$\dots \xrightarrow{N} \ZZ[C_n] \xrightarrow{t-1} \ZZ[C_n] \xrightarrow{N} \ZZ[C_n] \xrightarrow{t-1} \ZZ[C_n] \to 0 $$
is a projective resolution of the trivial module over $C_n,$ where $N=1+t+\dots + t^{n-1}.$ Then for any $\ZZ[C_n]$-module $M$ the complex
$$ \dots \xrightarrow{N} M \xrightarrow{t-1} M \xrightarrow{N} M \xrightarrow{t-1} M \to 0 $$
computes $H_*(C_n,M).$ If $M$ is a $k[C_n]$-module, we know that $H_n(C_n,M)=0$ for $n\ne 0$ and $H_0(C_n,M)=M_{C_n}$ because ${\sf char}(k)=0.$ It follows that there is a four term exact sequence
\begin{equation}\label{eq_M_exact_Q}
0 \longrightarrow M_{C_n} \xrightarrow{\tilde N} M \xrightarrow{t-1} M \longrightarrow M_{C_n} \longrightarrow 0,
\end{equation}
where $\tilde N$ is induced by $N.$

Since $d:F\to \Omega(F)$ is a derivation, we have $d([a,b])=[d(a),b]+[a,d(b)].$ It follows that $d([F,F])\subseteq [\Omega(F),F].$ Moreover, $d(1)=0.$ Then $\tilde d:\bar F_\natural \to \Omega(F)_\natural$ is well defined.
In order to prove the statement, it is enough to prove that the sequence
$$ 0 \longrightarrow \bar F_\natural \xrightarrow{\tilde d} \Omega(F)_\natural \xrightarrow{\tilde \alpha}  F \longrightarrow  F_\natural \longrightarrow 0$$
is exact, where $\tilde \alpha$ is the composition of $\alpha$ with the embedding $[F,F]\subset F.$ We are already know that the sequence $\Omega(F)_\natural \to F \to F_\natural \to 0$ is exact. So we need to prove that ${\rm Ker}(\tilde \alpha)={\rm Im}(\tilde d)$ and that $\tilde d$ is a monomorphism.

Let $F=T(V)$ be the tensor algebra on a vector space $V.$ Then it is  well known that there is an isomorphism of bimodules.
\begin{equation}\label{eq_tvt}
F\otimes V \otimes F \cong \Omega(F),
\end{equation}
$$ a \otimes v \otimes b \mapsto av\otimes b - a\otimes vb$$
(for example, see \cite[3.1.3]{Loday}; the proof of \cite[Prop. 9.1.6]{Weibel}; or \cite[Prop. 5.8]{QuillenC}).
Since $F\otimes V \otimes F$ is a free bimodule, and there are isomorphisms $(-)_\natural\cong (-) \otimes_{F^e}F$ and $(F\otimes V \otimes F)\otimes_{F^e}F=V\otimes F^e\otimes_{F^e}F=V\otimes F$, we have an isomorphism
\begin{equation}\label{eq_vt}
(F\otimes V \otimes F)_\natural = V\otimes F,
\end{equation}
$$a\otimes v \otimes b + [F\otimes V \otimes F,F] \leftrightarrow v\otimes ba.$$
The composition of the isomorphism $V\otimes F\cong \Omega(F)_\natural$ together with the map $\tilde \alpha$ is the map
$$\tilde \alpha':  V\otimes T(V) \longrightarrow T(V), $$ $$\tilde \alpha'(v_0\otimes v_1\dots v_n)=[v_1\dots v_n,v_0]=v_1\dots v_nv_0-v_0v_1\dots v_n.$$
Note that $\tilde \alpha'=\oplus \tilde \alpha'_n,$ where $$\tilde \alpha'_n:V\otimes V^{\otimes n-1} \longrightarrow V^{\otimes n}, \hspace{1cm} \tilde \alpha'_n=\cdot (t-1),$$
where $t$ is given by the obvious action of $C_n$ on $V^{\otimes n}.$ Since the sequence $\Omega(F)_\natural \xrightarrow{\tilde \alpha} F \to F_\natural \to 0$ is exact, we obtain $F_\natural= k \oplus \left(\bigoplus_{n\geq 1} (V^{\otimes n})_{C_n}\right).$ Therefore $\bar F_\natural= \bigoplus_{n\geq 1} (V^{\otimes n})_{C_n}.$ The exact sequence \eqref{eq_tvt} implies the exact sequence
$$0 \longrightarrow (V^{\otimes n})_{C_n} \xrightarrow{\tilde N} V^{\otimes n} \xrightarrow{t-1} V^{\otimes n} \longrightarrow (V^{\otimes n})_{C_n} \longrightarrow 0.$$
Then we only need to prove that $\tilde d:\bar F_\natural \to \Omega(F)_\natural$ composed with the isomorphisms $\bar F_\natural=\bigoplus_{n\geq 1} (V^{\otimes n})_{C_n}$ and $\Omega(F)_\natural= \bigoplus_{n\geq 1} V\otimes V^{\otimes n-1}$ is given by the homomorphism $\tilde N$ on each direct summand.

Indeed. The composition of the universal derivation $d:F\to \Omega(F)$ with the isomorphism  $\Omega(F)=T(V)\otimes V \otimes T(V)$ is given by
$$d':T(V) \longrightarrow T(V) \otimes V \otimes T(V) $$
$$d'(v_1  \dots  v_n)=\sum v_1\dots v_{i-1} \otimes v_i \otimes v_{i+1}\dots v_n.$$ If we pass to the quotients $T(V)_\natural$ and $(T(V)\otimes V \otimes T(V))_\natural$ and compose it  with the isomorphism \eqref{eq_vt} we obtain the map
$$\tilde d':T(V)_\natural \longrightarrow V\otimes T(V)$$
given by
$$\tilde d'(v_1  \dots  v_n+[T(V),T(V)])=\sum   v_i \otimes v_{i+1}\dots v_nv_1\dots v_{i-1},$$
which is equal to $\tilde N$.
\end{proof}

\begin{Remark} If ${\sf char}(k)\ne 0,$ the isomorphism $HH_1(F)=\bar F_\natural$ fails. This can be shown using the Connes' exact sequence for $F$ and the computation of $HC_*(F)$ (\cite[Theorem 3.1.6]{Loday}).
\end{Remark}

\begin{Theorem}[Hopf's formula for $\overline{HC}_{2n+1}$]\label{th_hopf_red} Let ${\sf char}(k)=0.$ Then for any $n\geq 0$ there is a natural isomorphism
$$ \overline{HC}_{2n+1}(A)=\frac{R^{n+1}\cap ([F,F]+k\cdot 1)}{[R,R^n]}.$$
Moreover, the functors $([F,F]+k\cdot 1)/[R,F]$ and $(R^{n+2} \cap [R,R^n])/[R,R^{n+1}]$ are constant  for $n\geq 0$.
\end{Theorem}
\begin{proof}
Quillen proved \cite[Th. 5.11]{QuillenC} that there is an exact sequence
$$0 \longrightarrow \overline{HC}_{2n+1}(A)\longrightarrow \frac{R^{n+1}}{[R,R^n]} \xrightarrow{\ \delta\ } \Omega(F)_\natural,$$
where $\delta$ is induced by the universal derivation $d:F\to \Omega(F).$ Lemma \ref{lemma_Omega(F)_decomposition}
 implies that $\delta$
 factors through the monomorphism $\bar F_\natural \xrightarrow{\tilde d} \Omega(F)_\natural.$ Therefore $\overline{HC}_{2n+1}(A)$
is the kernel of the map $R^{n+1}/{[R,R^{n+1}]} \to  F/([F,F]+k\cdot 1).$ Then $\overline{HC}_{2n+1}(A)= ({R^{n+1}\cap ([F,F]+k\cdot 1)})/{[R,R^n]}.$ If we take $n=0,$ we obtain a short exact sequence
$$0 \longrightarrow \overline{HC}_{1}(A) \longrightarrow \frac{[F,F]+k\cdot 1}{[R,F]} \longrightarrow [A,A]+k\cdot 1 \longrightarrow 0.$$
Since $\overline{HC}_{1}(A)$ and $[A,A]+k\cdot 1$ are constant, the middle term is constant as well.
The functor $(R^{n+2} \cap [R,R^n])/[R,R^{n+1}]$ is the kernel of the morphism of constant functors $\overline{HC}_{2n+3}(A)\to \overline{HC}_{2n+1}(A)$ and hence it is constant.
\end{proof}

\begin{Proposition}\label{prop_[F,F][R,F][R,R]}
Let ${\sf char}(k)=0.$ Then there are isomorphisms
$$\overline{HC}_{n+1}(A)=\colim_{n}\ [F,F]=\colim_{n} \  [R,F], $$
$$ \colim_0\ [R,F]=0, $$
for $n\geq 1,$ and the embedding $[R,F]\subseteq [F,F]$ induces isomorphims $\colim_{n}\ [F,F]=\colim_{n}\ [R,F]$ for $n\geq 1.$ Moreover,
$$ \overline{HC}_{n+2}(A)= \colim_{n}\ H_1(F,R) $$
for $n\geq 0,$ and
$$ \overline{HC}_{n+3}(A)= \colim_{n}\ [R,R]$$
for $n\geq 1.$
\end{Proposition}
\begin{proof}
The short exact sequence $[F,F]+k\cdot 1 \mono F \epi \bar F_\natural$ implies that $\colim_n ([F,F]+k\cdot 1)=\colim_{n+1} \bar F_\natural$ for $n\geq 1$ because $\colim_n\: F=0.$ Theorem \ref{theorem_cyclic} implies $\colim_n ([F,F]+k\cdot 1)=HC_{n+1}(A)$ for $n\geq 1.$ The short exact sequence $[F,F]\mono [F,F]+k\cdot 1 \epi k$ implies $\colim_n [F,F]=HC_{n+1}(A)$ for $n\geq 1.$ Since the functor $[F,F]/[R,F]$ is constant (Theorem \ref{th_hopf_red}), the short exact sequence $[R,F]\mono [F,F]\epi [F,F]/[R,F]$ implies that the embedding $[R,F] \mono [F,F] $ induces isomorphisms $\colim_n\  [R,F]= \colim_n\ [F,F] $ for $n\geq 1.$ Then $$\colim_n\ [R,F]=\overline{HC}_{n+1}(A)$$ for $n\geq 1.$

Since $\Omega(F)$ is a free $F$-bimodule, we have ${\sf Tor}^{F^e}_m(\Omega(F),R)=0$ for $m\ne 0.$ Then Proposition \ref{prop_O-bimodules-tor} implies that $\colim_n\: \Omega(F)\otimes_{F^e}R=0$ for any $n\in \mathbb Z.$ Hence the short exact sequence $H_1(F,R)\mono \Omega(F)\otimes_{F^e}R \epi [R,F]$ implies
$$ \colim_{n} \: H_1(F,R)=\colim_{n+1}\: [R,F]$$
for any $n\in \mathbb Z.$ It follows that 
$$\colim_{n} \ H_1(F,R)=\overline{HC}_{n+2}(A)$$
for $n\geq 0$ and $\colim_0 [R,F]=0.$

Lemma \ref{lemma_R^n} implies that $\colim_* R^2=0.$ Then the short exact sequence $[R,R]\mono R^2 \epi R^2/[R,R]$ implies $\colim_n [R,R]=\colim_{n+1} \ R^2/[R,R] $ for all $n.$
Quillen proved \cite[Th. 5.11]{QuillenC} that there is an exact sequence
$$0 \longrightarrow \overline{HC}_{3}(A) \longrightarrow \frac{R^2}{[R,R]} \longrightarrow H_1(F,R) \longrightarrow \overline{HC}_{2}(A) \longrightarrow 0.	$$ Applying the spectral sequence of colimits (Proposition \ref{prop_spectral_sequence_of_colimits}) to this small complex, and using that this spectral sequence converges to zero, we obtain that $\colim_{n+1} \ R^2/[R,R]=\colim_{n+1}\ H_1(F,R)$ for $n\geq 1.$ Therefore, for $n\geq 1$ we get $\colim_n [R,R]=\overline{HC}_{n+3}(A).$
 \end{proof}

\subsection{Connes' exact sequence via derived colimits}

In this subsection we give two ways how to obtain the Connes' exact sequence from the developed theory.
Here we assume that ${\sf char}(k)=0.$

\begin{Proposition}\label{prop_connes_omega} Let ${\sf char}(k)=0.$ Then the  long exact sequence of derived colimits of the short exact sequence
$$ 0 \longrightarrow \bar F_\natural \longrightarrow \Omega(F) \longrightarrow [F,F] \longrightarrow 0 $$
(from Lemma \ref{lemma_Omega(F)_decomposition}) gives a  long exact sequence $$ \dots \to \overline{HC}_4 \to \overline{HC}_2 \to HH_3 \to \overline{HC}_3 \to  \overline{HC}_1 \to HH_2 \to \overline{HC}_2,$$
where $\overline{HC}_n=\overline{HC}_n(A)$ and $HH_n=HH_n(A).$
\end{Proposition}
\begin{proof}
This follows from Theorem \ref{theorem_cyclic}, Theorem \ref{theorem_hochsch} and Proposition \ref{prop_[F,F][R,F][R,R]}.
\end{proof}

\begin{Lemma}\label{lemma_[R,F]/[R,R]} Let ${\sf char}(k)=0.$ Then for $n\geq 2$ there is an isomorphism
$$\colim_n\ \frac{[R,F]}{[R,R]}=HH_{n+2}(A).$$
\end{Lemma}
\begin{proof}
Consider the short exact sequence
$$ 0 \longrightarrow \frac{R^2\cap [R,F]}{[R,R]} \longrightarrow
\frac{[R,F]}{[R,R]} \longrightarrow \frac{[R,F]}{R^2\cap [R,F]} \longrightarrow 0.$$
The first term is constant by  Theorem \ref{th_hopf_red}. The last term is equal to $(R^2+[R,F])/R^2$ by the third isomorphism theorem. Using that  $\colim_*\ R^2=0$ we obtain, $\colim_n [R,F]/[R,R]=\colim_n R^2 +[R,F]$ for $n\geq 2.$ Then the assertion follows from Corollary \ref{cor_R^2+[R,F]}.
\end{proof}

\begin{Proposition}\label{connesseq} Let ${\sf char}(k)=0.$ Then the  long exact sequence of derived colimits of the short exact sequence
$$ 0 \longrightarrow [R,R] \longrightarrow [F,F] \longrightarrow [F,F]/[R,R] \longrightarrow 0   $$
gives a  long exact sequence
$$ \dots \to \overline{HC}_6 \to \overline{HC}_4 \to HH_5 \to \overline{HC}_5 \to  \overline{HC}_3 \to HH_4 \to \overline{HC}_4\to \overline{HC}_2 $$
\end{Proposition}
\begin{proof}
This follows from Proposition \ref{prop_[F,F][R,F][R,R]} and Lemma \ref{lemma_[R,F]/[R,R]}.
\end{proof}

\section{\bf Cyclic homology of non-unital algebras}

In this section we consider the category of non-unital algebras ${\sf Alg}^{\sf n}$ and present non-unital versions of some results about cyclic homology from the previous section. We assume that ${\sf char}(k)=0.$  Effective epimorphisms of this category are surjective homomorphisms. Projective objects of this category are retracts of free non-unital algebras, where free non-unital algebras can be described as reduced tensor algebras $F=\bar T(V)=\bigoplus_{n\geq 1} V^{\otimes n}.$ As in the previous section Proposition \ref{proposition_subcategory_retract} allows us to consider only presentations $F\epi A,$  where $F$ is a free algebra.  We set $F_\natural = HC_0(F)=F/[F,F]$ and $R={\rm Ker}(F\epi A).$

\begin{Theorem}[{cf. \cite{DonadzeInassaridzeLadra}}] Assume that ${\sf char}(k)=0.$ Then for any $n\geq 0$ there is an isomorphism
$$\colim_n\  F_\natural= HC_n(A).$$
\end{Theorem}
\begin{proof}
The proof is similar to the proof of Theorem \ref{theorem_cyclic}.
\end{proof}

\begin{Theorem}[Hopf's formula for $HC_{2n+1}$]\label{th_hopf} Let ${\sf char}(k)=0.$ Then for any $n\geq 0$ there is a natural isomorphism
$$ HC_{2n+1}(A)=\frac{R^{n+1}\cap [F,F]}{[R,R^n]}.$$
\end{Theorem}
\begin{proof}
We denote by $A_+$ the algebra with added formal unit $A_+=A\oplus k.$ It is well known that $HC_*(A)=\overline{HC}_*(A_+).$ Theorem \ref{th_hopf_red} implies that
$$ HC_{2n+1}(A)=\frac{R^{n+1}\cap ([F_+,F_+]+k\cdot 1)}{[R,R^n]}.$$
Then the assertion follows from the equations $[F_+,F_+]=[F,F]$ and $F\cap ([F,F]+k\cdot 1)=[F,F].$
\end{proof}

\section{\bf K-functors}

Recall that for any unital ring $A$, there is a notion of
its Steinberg group $St(A)$.
Some of its properties are listed here:
\begin{itemize}
\item
$St(A)$ is the quotient of a free group on generators $e_{i,j}(x)$ for integers $i\neq j$ and $x\in A$ modulo
relations
\begin{align*}
& e_{i,j}(x)e_{i,j}(y)= e_{i,j}(x+y),\\
& [e_{i,j}(x), e_{j,k}(y)] = e_{i,k}(xy) \text{ if } i\neq k,\\
& [e_{i,j}(x), e_{i',j'}(y)] = 1 \text{ if } i\neq j' \text{ and } j \neq i'.
\end{align*}
\item
$H_1(St(A))=H_2(St(A))=0.$

\item
$H_3(St(A))=K_3(A)$ (see \cite{gersten}).

\item
There is an exact sequence of groups
$1 \to K_2(A) \to St(A) \to E(A) \to 1$
and moreover it is the universal central extension of $E(A):=[GL(A), GL(A)].$
\end{itemize}

For a ring homomorphism $A \stackrel{f}\to B$ we denote by 
$K_n(A,B)$ the $n$-th relative $K$-theory group, i.e. the $n$-th homotopy 
group of the homotopy fiber of the induced map of spectra $K(A) \to K(B)$.

Now let $k$ be a noetherian regular commutative ring and $A$ be an $k$-algebra. We denote by 
$\tilde{K}_n(A)$ the $n$-th reduced $K$-theory group, $K_{n-1}(k,A)$. By definition there is an exact sequence 
$$0 \to \frac{K_n(A)}{K_n(k)} \to \tilde{K}_n(A) \to {\rm Ker}(K_{n-1}(k) \to K_{n-1}(A)) \to 0.$$
If $A$ admits an augmentation $\tilde{K}_n(A)$ is just the quotient $\frac{K_n(A)}{K_n(k)}$. 

For any functor ${\sf Rings} \stackrel{\mathcal F}\to \mathcal{A}$ we denote by the same letter its extension to the category of
non-unital rings ${\sf Rngs}$ given by the formula $\mathcal F(R) = \frac{\mathcal F(R\times \mathbb{Z})}{\mathcal F(\mathbb{Z})} = {\rm Ker} (\mathcal F(R\times \mathbb{Z}) \to \mathcal F(\mathbb{Z})).$

We will need the following lemma well-known for specialists in the field. The essential ingredients for the lemma were given in Keune’s paper \cite{keune}. 
Note that in \cite[Def. IV.13.17]{magurn} the statement of the lemma is used 
as a definition of the relative $K_2$. 

\begin{Lemma}\label{relativeK2description}
Let $F \to A$ be a surjective ring homomorphism. 
Set $D = F \times_A F$. 

Then the group $K_2(F,A)$ is isomorphic to the quotient 
$$\frac{K_2(D)}{K_2(F) + \Gamma}$$
where $K_2(F)$ is considered as a subgroup of $K_2(D)$ via the diagonal split 
embedding $F \to D$ and $\Gamma$ is the subgroup of $St(D)$ generated by 
commutators 
$[e_{1,2}(x,0),e_{2,1}(0,y)]$ for all $x,y \in {\rm Ker}(F \to A)$. 
\end{Lemma}
\begin{proof}
For any surjective ring homomorphism $F \to A$ there is a group $St(F,A)$ and 
a map $St(F,A) \to {\rm Ker}(E(F) \to E(A))$ such that $K_2(F,A)$ is the kernel of the map 
(see \cite[\S5]{keune}). 
By Theorem 12 of ibid. the group is computed as 
$$St(F,A) = \frac{H}{\Gamma^H}$$ 
where $H$ is the kernel of the projection $St(D) \to St(F)$ onto 
the first component. 
Note that $E(F\times_A F) = E(F) \times_{E(A)} E(F)$, so every generator of 
$\Gamma$ 
is sent to a trivial element via the map $H \to E(D)$. 
This shows that the map 
$$St(F,A) \to {\rm Ker}(E(D)\to E(F)) = {\rm Ker}(E(F)\to E(A))$$ 
is well-defined. 
Moreover, this tells us that $\Gamma$ is actually a subgroup of $K_2(D)$ 
and that $St(F,A) = H/\Gamma$ (since $K_2$ is central by \cite[Pr. 13]{keune}).

Now the natural map 
${\rm Ker} (K_2(D) \to K_2(F)) \stackrel{f}\to H$ 
induces a map on quotients 
$$\frac{{\rm Ker}(K_2(D) \to K_2(F)}{\Gamma} \to \frac{H}{\Gamma}.$$ 
The map is injective since the map $K_2(D) \to St(D)$ is. 
Moreover, an easy diagram chasing argument shows that it surjects onto the kernel of the map from $H/\Gamma$ to $E(F\times_A F)$. 

Indeed, we have a commutative diagram 
$$
\begin{tikzcd}
\frac{{\rm Ker}(K_2(D) \to K_2(F))}{\Gamma}\arrow[r, hook] & \frac{H}{\Gamma}\arrow[ddr]\\
{\rm Ker}(K_2(D) \to K_2(F))\arrow[r,hook]\arrow[d,hook]\arrow[u,two heads] & H\arrow[d, hook]\arrow[u, two heads]\arrow[dr] \\
K_2(D) \arrow[r, hook] & St(D)\arrow[r, two heads] & E(D)
\end{tikzcd}
$$
whose lower row is exact. 
Let 
$x \in H/\Gamma$ be an element mapped to a trivial elementary matrix. Denote 
its lift to $H$ by $\tilde{x}$. By exactness of the lower row the image of $\tilde{x}$ in $St(D)$ 
belongs to $K_2(D)$. It is also mapped to a trivial element in $St(F)$, therefore, it gives rise to an element $y$ of ${\rm Ker} (K_2(D) \to K_2(F))$. It is  mapped back to $\tilde{x}$ by $f$ since the map $H \to St(D)$ is an inclusion. Now the image of $y$ under the upper left vertical map is an element in the preimage of $x$. 

Hence, we obtain an isomorphism 
$$\frac{{\rm Ker}(K_2(D) \to K_2(F))}{\Gamma} \cong K_2(F,A).$$
Lastly, the diagonal map $F \to D$ gives a splitting of the map 
$K_2(D) \to K_2(F)$ and an isomorphism 
$$\frac{{\rm Ker}(K_2(D) \to K_2(F))}{\Gamma} \cong \frac{K_2(D)/K_2(F)}{\Gamma}= \frac{K_2(D)}{K_2(F)+\Gamma}.$$

\end{proof}

\begin{Lemma}\label{colimGamma}
Let $A$ be a ring without unit or a $k$-algebra. 
Consider $\Gamma$, the subgroup of $St(F \times_A F)$ generated by
commutators $[e_{1,2}(x,0),e_{2,1}(0,y)]$ for all $x,y \in {\rm Ker}(F \to A)$, as a 
functor on the category of presentations of $A$. 

Then $\colim\ \Gamma = 1$, where the colimit is taken over the category of 
presentations in the category of $k$-algebras (resp. the category 
${\sf Rngs}$). 
\end{Lemma}
\begin{proof}
By Proposition 2.13, it suffices to show that for a fixed presentation $F \stackrel{p}\to A$ the coequalizer of all the maps $\Gamma(F) \to \Gamma(F)$ 
induced by maps of presentations $F \to F$ is trivial. 

Choose a set $S$ such that $F = k\langle S\rangle$ (resp. $F = F\langle S \rangle$). 
Adding a new variable if necessary we can assume that there is an element $s\in S$ such that $p(s) = 0$.
For an element $a \in {\rm Ker}(F\to A)$ consider the map of sets
$S \to F$ that is identitical on $S - \{ s_0\}$ and sends $s_0$ to $a$.
By the universal property of a free algebra (resp. a free ring) the map 
extends uniquely to a map of rings
$F\stackrel{f_r}\to F$.
The equality $p \circ f_r(s) = p(s)$ holds for any $s\in S$ by design, hence 
$p\circ f_r = p$ by the universal property of a free algebra (resp. a free 
ring).

The map $(f_r)_*$ sends $[e_{1,2}(s_0,0), e_{2,1}(0,b)]$ to $[e_{1,2}(a,0), e_{2,1}(0,b)]$ while the map $(f_0)_*$ sends the same element to 
$[e_{1,2}(0,0), e_{2,1}(0,b)] = 1$ 
for any $b \in k\langle S-\{s_0\} \rangle$ (resp. $F\langle S-\{s_0\} \rangle$). 
For arbitrary $b$, moreover, $(f_0)_*$ sends $[e_{1,2}(a,0), e_{2,1}(0,b)]$ to 
$[e_{1,2}(\tilde{a},0), e_{2,1}(0,\tilde{b})]$, where $\tilde{b} \in k\langle S-\{s_0\} \rangle$ (resp. $F\langle S-\{s_0\} \rangle$).
Therefore all generators of $\Gamma(F)$ become trivial in the coequalizer and 
the coequalizer itself is trivial.
\end{proof}

\begin{Proposition}\label{colimK2}
1. Let $A$ be a $k$-algebra.
Then
$$\colim\ \tilde{K}_2(F\times_A F) = \tilde{K}_3(A),$$
where the colimit is taken over the category of presentations in the category of $k$-algebras.

2. Let $A$ be a non-unital ring. Then
$$\colim\ K_2(F\times_A F) = K_3(A),$$
where the colimit is taken over the category of presentations in the category ${\sf Rngs}$.
\end{Proposition}
\begin{proof}
Let $A$ be a $k$-algebra and let $F \to A$ be a surjective ring homomorphism where $F$ is a free $k$-algebra. 
By Lemma \ref{relativeK2description} we have an exact sequence 
$$\Gamma \to K_2(F\times_A F)/K_2(F) \to K_2(F,A) \to 0$$
where $\Gamma$ is the subgroup of $St(F\times_A F)$ generated by by 
commutators 
$[e_{1,2}(x,0),e_{2,1}(0,y)]$ for all $x,y \in {\rm Ker}(F \to A)$. 

By \cite[Cor. 3.9]{gersten2}, the natural map $K_2(k) \to K_2(F)$ is an isomorphism, so by 5-Lemma $\tilde{K}_3(A)=K_2(k,A) \to K_2(F,A)$ is also an isomorphism. 
Hence the exact sequence above is isomorphic to the exact sequence 
$$\Gamma \to \tilde{K}_2(F\times_A F) \to \tilde{K}_3(A) \to 0.$$
Now Lemma \ref{colimGamma} and right-exactness of the colimit functor imply
the first part of the statement. 
The same argument applied to $F_+ \to A_+$, where $F \to A$ is a surjection from a free ring to a non-unital ring $A$,  
yields the second part of the statement.
\end{proof}

\newpage

\centering \includegraphics[width=0.35\textwidth]{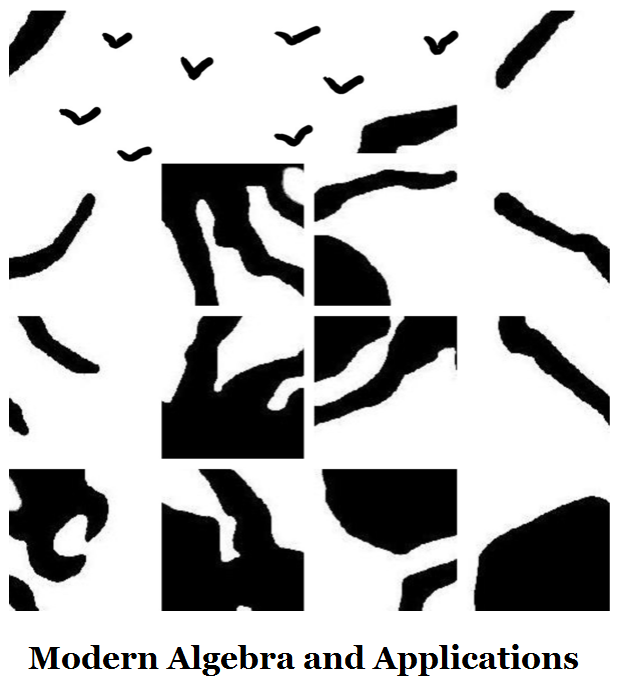}

\end{document}